\documentclass{article}

\usepackage{amsfonts,amssymb,amsmath,graphicx,authblk,rotating}
\usepackage[margin=25mm]{geometry}
\providecommand{\keywords}[1]
{
  \small	
  \textbf{\textit{Keywords -- }} #1
}
\newtheorem{theorem}{Theorem}[section]

\newtheorem{lemma}{Lemma}[section]

\newtheorem{remark}{Remark}[section]
\newtheorem{hp}{Assumption}[section]
\newenvironment{proof}{\noindent{\bf Proof.\hskip.5em}\ignorespaces}{}
\usepackage[numbers]{natbib}

\usepackage{subcaption}
\usepackage{bm}
\usepackage{mathrsfs}
\usepackage{tikz}
\usepackage{lscape}
\usepackage{cases}
\usepackage{mathtools}
\usepackage{cancel}
\usepackage[inter-unit-product=\cdot,
            separate-uncertainty=true,
            print-unity-mantissa=false,
            table-alignment-mode=none,
            table-alignment=right,
            exponent-product=\cdot]{siunitx}

\usepackage{hyperref}
\hypersetup{colorlinks, linkcolor=black, citecolor=red, urlcolor=blue}
\usepackage[capitalize]{cleveref}

\Crefname{hp}{Assumption}{Assumptions}

\newcommand{\lymph}{\texttt{lymph}}
\newcommand{\bdout}{{\rm w}}

\newcommand{\PP}{{\rm el}}
\newcommand{\FF}{{\rm f}}
\renewcommand{\vec}[1]{{\bm{#1}}}

\newcommand{\jj}{{\rm j}}
\newcommand{\kk}{{\rm k}}
\newcommand{\EE}{{\rm E}}
\newcommand{\II}{{\rm I}}
\newcommand{\DD}{{\rm D}}
\newcommand{\NN}{{\rm N}}

\newcommand{\normcoerc}[1]{%
  |\mkern-1.5mu|
   #1
  |\mkern-1.5mu|
}
\newcommand{\normDGd}[1]{\normcoerc{#1}_{{\rm DG},{\rm D}}}
\newcommand{\normDGj}[1]{\normcoerc{#1}_{{\rm DG},{\rm P}_\jj}}
\newcommand{\normDGk}[1]{\normcoerc{#1}_{{\rm DG},{\rm P}_\kk}}
\newcommand{\normDGE}[1]{\normcoerc{#1}_{{\rm DG},{\rm P}_\EE}}
\newcommand{\normDGu}[1]{\normcoerc{#1}_{{\rm DG},{\rm U}}}
\newcommand{\normDGp}[1]{\normcoerc{#1}_{{\rm DG},{\rm P}_\FF}}
\newcommand{\normENporoel}[1]{\normcoerc{#1}_{\PP,t}}
\newcommand{\normENfluid}[1]{\normcoerc{#1}_{\FF,t}}
\newcommand{\normEN}[1]{\normcoerc{#1}_{{\rm EN},t}}
\newcommand{\normENzero}[1]{\normcoerc{#1}_{{\rm EN},0}}

\newcommand{\normcont}[1]{%
  |\mkern-1.5mu|\mkern-1.5mu|
   #1
  |\mkern-1.5mu|\mkern-1.5mu|
}
\newcommand{\normcontD}[1]{\normcont{#1}_{{\rm D}}}
\newcommand{\normcontJ}[1]{\normcont{#1}_{{\rm P}_\jj}}
\newcommand{\normcontK}[1]{\normcont{#1}_{{\rm P}_\kk}}
\newcommand{\normcontU}[1]{\normcont{#1}_{{\rm U}}}
\newcommand{\normcontP}[1]{\normcont{#1}_{{\rm P}_\FF}}

\newcommand{\Div}{{\rm div}}

\newcommand{\averagel}{\{\!\!\{}
\newcommand{\averager}{\}\!\!\}}

\newcommand{\jumpl}{[\![}
\newcommand{\jumpr}{]\!]}

\newcommand{\average}[1]{\averagel#1\averager}

\newcommand{\jump}[1]{\jumpl#1\jumpr}

\newcommand{\spaceW}{\vec{W}^{\rm DG}_h}
\newcommand{\spaceV}{\vec{V}^{\rm DG}_h}
\newcommand{\spaceQ}{Q^{\rm DG}_h}
\newcommand{\spaceQj}{Q^{\rm DG}_{\jj,h}}
\newcommand{\spaceQk}{Q^{\rm DG}_{\kk,h}}
\newcommand{\spaceQE}{Q^{\rm DG}_{\EE,h}}

\newcommand{\cc}{\cos(\pi x)\cos(\pi y)}

\renewcommand{\ss}{\sin(\pi x)\sin(\pi y)}

\DeclareSIUnit\mmhg{mmHg}

\title{Polytopal discontinuous Galerkin discretization of brain multiphysics flow dynamics}
\author{Ivan Fumagalli, Mattia Corti, Nicola Parolini, Paola F. Antonietti}
\affil{\small MOX, Department of Mathematics, Politecnico di Milano, piazza Leonardo da Vinci 32, Milan, 20133, Italy}
\date{}

\begin{document}

\maketitle

\begin{abstract}
A comprehensive mathematical model of the multiphysics flow of blood and Cerebrospinal Fluid (CSF) in the brain can be expressed as the coupling of a poromechanics system and Stokes' equations: the first describes fluids filtration through the cerebral tissue and the tissue's elastic response, while the latter models the flow of the CSF in the brain ventricles.
This model describes the functioning of the brain's waste clearance mechanism, which has been recently discovered to play an essential role in the progress of neurodegenerative diseases.
To model the interactions between different scales in the porous medium, we propose a physically consistent coupling between Multi-compartment Poroelasticity (MPE) equations and Stokes' equations.
In this work, we introduce a numerical scheme for the discretization of such coupled MPE-Stokes system, employing a high-order discontinuous Galerkin method on polytopal grids
to efficiently account for the geometric complexity of the domain.
We analyze the stability and convergence of the space semidiscretized formulation, we prove a-priori error estimates, and we present a temporal discretization based on a combination of Newmark's $\beta$-method for the elastic wave equation and the $\theta$-method for the other equations of the model.
Numerical simulations carried out on test cases with manufactured solutions validate the theoretical error estimates.
We also present numerical results on a two-dimensional slice of a patient-specific brain geometry reconstructed from diagnostic images, to test in practice the advantages of the proposed approach.
\end{abstract}


\section{Introduction}\label{sec:intro}

In the brain, multiple fluid components play different roles: the blood supplies nutrients and oxygen and removes carbon dioxide, while the Cerebrospinal Fluid (CSF) and the glymphatic system \cite{glymphatic1,glymphatic2,glymphatic3} have the primary function of clearing the waste produced by brain activity.
Also, it has been recently shown that waste clearance mechanisms play a major role in the evolution of neurodegenerative diseases \cite{bacyinski2017paravascular,kylkilahti2021achieving,gouveia2021perivascular,brennan2023role}.
These physical systems are strongly interconnected with one another and with the cerebral matter: for example, a large amount of the CSF is generated in the choroid plexus thanks to the high concentration of blood capillary vessels in it \cite{csfGeneration}; also, an auxiliary function of the CSF is to protect the brain from impact against the skull and to compensate blood pulsatility in terms of flow rate and pressure inside the braincase.
For this reason, the modeling of the fluid dynamics in the brain requires a multiphysics perspective, able to capture these interactions and their mutual interplay.

From the modeling viewpoint, the brain can be described as a porous material (white and gray matter) with fluids both filtrating through it and flowing in hollow regions (brain ventricles).
In terms of mathematical modeling, this can be represented as the coupling between a poromechanics system -- described, e.g., by Darcy's or Biot's equations \cite{biotZienkiewicz}
-- and Stokes/Navier-Stokes' equations for the flow in the ventricles \cite{linninger2016cerebrospinal,gholampour2023mathematical}.
Numerical methods for such Fluid-Poroelastic Structure Interaction (FPSI) systems can be found in the literature: the most studied is the Biot-Stokes' system
\cite{badia2009coupling,zunino2018biot,ager2019biotLITREV,boon2022biot}, 
but also more complex models have been investigated, e.g.~considering a multilayer structure of the porous medium \cite{bociu2021multilayered,bukavc2015multilayered}, a non-linear constitutive model for either the structural or fluid part of the system \cite{dereims20153d},
or 
Multi-compartment Poroelasticity (MPE) models coupled with Stokes' equations \cite{digregorio2021computational,barnafi2022multiscale}.
This latter framework is the best suited to model brain poromechanics and CSF flow, since it can simultaneously represent cerebral tissue deformation, blood vessel networks at different scales, CSF flow, and waste clearance.
Moreover, the dynamic formulation of MPE equations can describe the inertial forces associated with blood pressure pulsatility during a heartbeat, which affects vascular and tissue deformations and waste clearance \cite{tully2011cerebral,chou2016fully,daversin2020mechanisms,corti2022numerical}.

The analysis of conforming Finite Element methods for FPSI problems has been addressed in several of the abovementioned works, with a discussion on the inf-sup requirements on the different components of the system (poroelastic, fluid, and coupling conditions).
However, aiming at an accurate representation of the poromechanics and flow with low dispersion and dissipation errors, high-order discretization methods provide a better choice \cite{quarteroni2009numerical,godlewski2013numerical}.
Polytopal Discontinuous Galerkin (PolyDG) methods fit in this framework, and they can also naturally account for complex geometries such as vessel networks and brain folds, thanks to their flexibility in terms of local refinement and agglomeration, hanging nodes treatment, and generality of mesh element shapes \cite{bassi2012flexibility,cangiani2014hp,antonietti2016review,cangiani2017book,pazner2018convergence,cangiani2022hp}.
In addition, discontinuous Galerkin schemes provide a general framework to embed physically-consistent interface conditions directly in the weak form.
So far, these methods have been mostly employed to solve porous media and poroelasticity problems for geophysical applications \cite{lipnikov2014discontinuous,antonietti2021high,antonietti2022high}
and for fluid dynamics and fluid-structure interaction problems \cite{cockburn2002local,di2010discrete,wirasaet2014discontinuous,antonietti2019numerical,ye2021conforming,zonca2021polygonal,AMVZ22}.

In this work, we introduce a high-order PolyDG method to spatially discretize a coupled model encompassing dynamic Multiple-Network Poroelastic (MPE) equations for poromechanics \cite{corti2022numerical} and Stokes' equations \cite{antonietti2023discontinuous,AMVZ22}, with physically-consistent coupling conditions inspired from the mass and stress balance at the interface between the brain tissue and the CSF.
Moreover, we take full advantage of the DG framework to embed, directly in the physically consistent formulation, the coupling conditions at the interface between the poroelastic and fluid regions.
We analyze the well-posedness and convergence of the semidiscrete numerical method, and we employ a combination of Newmark's $\beta$-method and the $\theta$-method for time discretization.

The paper is organized as follows.
In \cref{sec:model}, we introduce the multiphysics problem, both in strong and weak form, and discuss the coupling conditions.
\cref{sec:polydg} describes the PolyDG space discretization; stability and convergence properties are analyzed in \cref{sec:apriori}.
Then, time discretization is introduced \cref{sec:fullydiscrete}.
Verification tests are discussed in \cref{sec:results}, corroborating the theoretical results of \cref{sec:apriori}.
\cref{sec:brain} demonstrates the capabilities of the method on a two-dimensional brain section considering physiological settings.

\section{Mathematical model}\label{sec:model}

We introduce the mathematical model consisting of the coupling between a Multiple-Network Poroelasticity system and Stokes' system: the former, introduced in \cite{corti2022numerical}, accounts for the poromechanics of the brain tissue and its interaction with different fluids compartments flowing in its pores, while the latter describes the flow of the CSF in the brain ventricles.
To ease the presentation, we consider a simplified configuration whose two-dimensional representation is depicted in \cref{fig:domain}.
The poroelastic medium occupies a portion $\Omega_\PP\subset\mathbb R^d$ ($d = 2,3$) of the domain, while the CSF flows according to Stokes' equations in the remaining portion $\Omega_\FF\subset\mathbb R^d$.
The interaction between the two systems occurs at the interface $\Sigma=\overline{\Omega}_\PP\cap\overline{\Omega}_\FF$, which we suppose to be a (piecewise) smooth $(d-1)-$manifold.
The source of CSF comes from an exchange of mass with the poroelastic medium and it exits from the domain at $\Gamma_\text{out}$, while the rest of the domain boundary is a solid wall $\Gamma_\bdout$.
We denote by $\Omega$ the interior of $\overline{\Omega}_\PP\cup\overline{\Omega}_\FF$.

\begin{figure}
    \centering
    \includegraphics[width=0.5\textwidth]{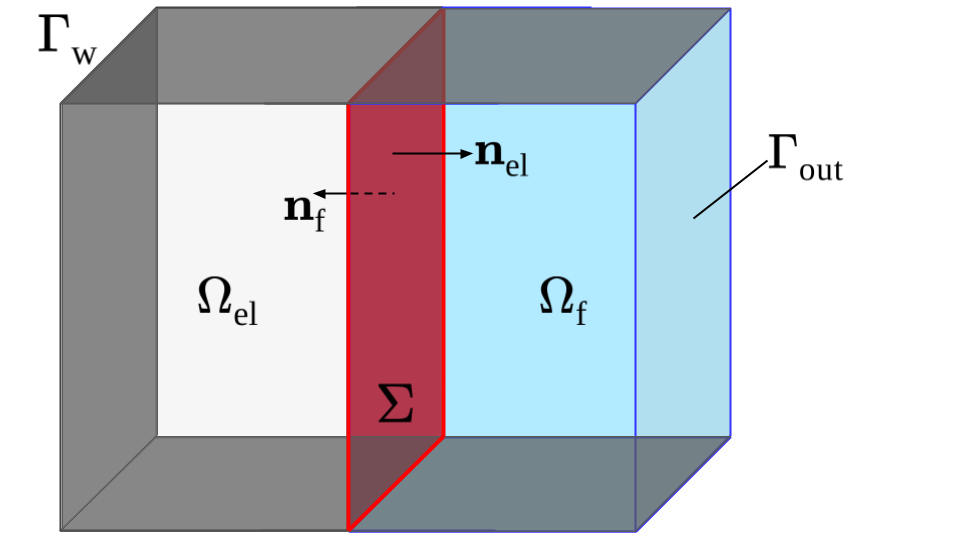}
    \caption{Domain scheme: poroelastic medium in $\Omega_\PP$ (light grey), Stokes' flow of CSF in $\Omega_\FF$ (blue), interface $\Sigma$ (red), and external boundaries $\Gamma_\bdout$ (dark grey) and $\Gamma_\text{out}$.}
    \label{fig:domain}
\end{figure}

Given a final observation time $T>0$, we introduce the CSF velocity $\vec{u}:\Omega_\FF\times[0,T]\to\mathbb  R^d$ and pressure $p:\Omega_\FF\times[0,T]\to\mathbb  R$, the solid tissue displacement $\vec{d}:\Omega_\PP\times[0,T]\to\mathbb R^d$, and the network pressures $p_\jj:\Omega_\PP\times[0,T]\to\mathbb R$, $\jj\in J$, where $J$ is a given set of labels denoting the different fluid network compartments \cite{corti2022numerical}.
In particular, for the application at hand, we consider $J = \{\text{A},\text{C},\text{V},\EE\}$, where A, C, V correspond to the arterial, capillary, and venous blood compartments, respectively, while $\EE$ denotes the extracellular CSF permeating the brain tissue.

The coupled problem reads as follows:
\begin{subnumcases}{\label{eq:NSMPE}}
    \rho_\PP\partial_{tt}^2\vec{d} - \nabla\cdot\sigma_\PP(\vec{d}) + \sum_{\kk\in J}\alpha_\kk\nabla p_\kk = \vec{f}_\PP,
    &$ \qquad\text{in } \Omega_\PP\times(0,T],$ \label{eq:elasticity}\\
    c_\jj\partial_t p_\jj+\nabla\cdot\left(\alpha_\jj\partial_t\vec{d}-\frac{1}{\mu_\jj}K_\jj\nabla p_\jj\right) \\
    \qquad+ \sum_{\kk\in J}\beta_{\jj\kk}(p_\jj-p_\kk) + \beta_\jj^\text{e}p_\jj = g_\jj,
    &$ \qquad\text{in } \Omega_\PP\times(0,T],\quad\forall\jj\in J,$ \label{eq:pj}\\
    \rho_\FF\partial_t\vec{u} - \nabla\cdot\sigma_\FF(\vec{u}) + \nabla p = \vec{f}_\FF,
    &$ \qquad\text{in } \Omega_\FF\times(0,T],$ \label{eq:fluidMom}\\
    \nabla\cdot\vec{u} = 0,
    &$ \qquad\text{in } \Omega_\FF\times(0,T],$ \label{eq:fluidCont},
\end{subnumcases}
where the linear elastic and fluid (viscous) stress tensors are defined as
$
\sigma_\PP(\vec{d}) = 2\mu_\PP\varepsilon(\vec{d})+\lambda(\nabla\cdot\vec{d})I
$ and $
\sigma_\FF(\vec{u}) = 2\mu_\FF\varepsilon(\vec{u})
$, respectively, with $
\varepsilon(\vec{w}) = \left(\nabla\vec{w}+\nabla\vec{w}^T\right)/2,
$.
The body forces $\vec{f}_\PP:\Omega_\PP\times(0,T]\to\mathbb R^d,g_\jj:\Omega_\PP\times(0,T]\to\mathbb R,\vec{f}_\FF:\Omega_\FF\times(0,T]\to\mathbb R^d$ are supposed sufficiently regular.
The parameters of the models are explained in \cref{tab:modelparams}, with their typical physiological values extracted from \cite{lee2019mixed,corti2022numerical,causemann2022human}.
The initial and boundary conditions of problem \eqref{eq:NSMPE} are defined as 
\begin{subnumcases}{\label{eq:bd}}
    \Bigl(\vec{d}(0), \partial_t\vec{d}(0)\Bigr)=\Bigl(\vec{d}_0,\dot{\vec{d}}_0\Bigr), \quad p_\jj(0)=p_{\jj0},
    &$ \qquad\text{in } \Omega_\PP,\quad\forall \jj\in J,$\\
    \vec{u}(0) = \vec{u}_0
    &$ \qquad\text{in } \Omega_\FF,$\\
    \vec{d} = \vec{d}_\DD = \vec{0},
    &$ \qquad\text{on } \Gamma_\bdout\times(0,T],$\\
    \frac{1}{\mu_\jj}K_\jj\nabla p_\jj\cdot\vec{n}_\PP = 0,
    &$ \qquad\text{on } \Gamma_\bdout\times(0,T],\quad\forall \jj\in J,$\\
    \vec{u} = \vec{u}_\DD = \vec{0},
    &$ \qquad\text{on } \Gamma_\bdout\times(0,T],$\\
    (\sigma_\FF(\vec{u}) - p I)\vec{n}_\FF = -\overline{p}^\text{out}\vec{n}_\FF,
    &$ \qquad\text{on } \Gamma_\text{out}\times(0,T],$
\end{subnumcases}
 with suitable definition of the data function $\overline{p}^\text{out}:\Gamma_\text{out}\times(0,T]\to\mathbb R$, that represents the external normal stress at the outlet, and of the initial conditions $\vec{d}_0:\Omega_\PP\to\mathbb R^d, \dot{\vec{d}}_0:\Omega_\PP\to\mathbb R^d, \vec{u}_0:\Omega_\FF\to\mathbb R^d, p_{\jj 0}:\Omega_\PP\to\mathbb R, \jj\in J$.

\begin{table}
    \centering
    \begin{tabular}{ccl}
        parameter & phys.~values & description\\
        \hline
        $\rho_\PP$ &   $\SI{1000}{\kilo\gram\per\cubic\meter}$ & density of the solid tissue \\
        $\rho_\FF$ &  $\SI{1000}{\kilo\gram\per\cubic\meter}$ & density of the CSF \\
        $\mu_\PP$ & $\SI{216}{\pascal}$ & first Lam\'e parameter of the solid \\
        $\lambda$ & $\SI{505}{\pascal}$ & second Lam\'e parameter of the solid \\
        $\mu_\jj$ & $\SI{3.5e-3}{\pascal\second}$ & viscosity of the fluid in compartment $\jj\in J$ \\
        $\mu_\FF$ & $\SI{3.5e-3}{\pascal\second}$ & viscosity of CSF \\
        $\alpha_\jj$ & $\in[0,1)$ & Biot-Willis coefficient of compartment $\jj\in J$ \\
        $c_\jj$ & $\SI{1e-6}{\square\meter\per\newton}$ & storage coefficient of compartment $\jj\in J$ \\
        $k_\jj$ & $\SI{1e-11}{\square\meter}$ & $K_\jj=k_\jj I$ permeability tensor for compartment $\jj\in J$ \\
        $\beta_{\jj\kk}$ & $\SI{1}{\square\meter\per\newton\per\second}$ & coupling transfer coefficient between compartments \\
        &&(from $\kk\in J$ to $\jj\in J$) \\
        $\beta_{\jj}^\text{e}$ & $\SI{1}{\square\meter\per\newton\per\second}$ & external coupling coefficient for compartment $\jj\in J$
    \end{tabular}
    \caption{Parameters of model \eqref{eq:NSMPE} with corresponding physiological values.}
    \label{tab:modelparams}
\end{table}

On the interface $\Sigma$, we introduce the following coupling conditions, based on physiological considerations:
\begin{subnumcases}{\label{eq:interf}}
    \sigma_\PP(\vec{d})\vec{n}_\PP - \sum_{\kk\in J} \alpha_\kk p_\kk\vec{n}_\PP + \sigma_\FF(\vec{u})\vec{n}_\FF - p \vec{n}_\FF = \vec{0}, 
    &$ \quad\text{on } \Sigma\times(0,T],$
    \label{eq:BCtotalstress}\\
    \frac{1}{\mu_\jj}K_\jj\nabla p_\jj\cdot\vec{n}_\PP = 0,
    &$ \quad\text{on } \Sigma\times(0,T],\quad\forall\jj\in J\setminus\{\EE\},$\label{eq:BCpnonE}\\
    \vec{u}\cdot\vec{n}_\FF + \left(\partial_t\vec{d}-\frac{1}{\mu_\EE}K_\EE\nabla p_\EE\right)\cdot\vec{n}_\PP = 0,
    &$ \quad\text{on } \Sigma\times(0,T],$\label{eq:BCnormalflux}\\
    p_\EE = p - \sigma_\FF(\vec{u})\vec{n}_\FF\cdot\vec{n}_\FF, 
    &$ \quad\text{on } \Sigma\times(0,T],$\label{eq:BCnormalstress}\\
    \left(\sigma_\FF(\vec{u})\vec{n}_\FF - p \vec{n}_\FF\right)\wedge\vec{n}_\FF = \vec{0}, &$ \quad\text{on } \Sigma\times(0,T].$\label{eq:BCtgstress}
\end{subnumcases}
Condition \eqref{eq:BCtotalstress} expresses the balance of total normal stress.
Due to the blood-brain barrier \cite{abbott2010structure,daneman2015blood}, we assume that mass exchange between the poroelastic domain and the CSF only occurs through compartment $\EE$, as expressed by \eqref{eq:BCpnonE}-\eqref{eq:BCnormalflux}.
Consistently, the normal stress of the CSF fluid is balanced by the pressure of compartment $\EE$, as in \eqref{eq:BCnormalstress}, while we assume the tangential stress on the fluid to be negligible (cf.~\eqref{eq:BCtgstress}).
Similar assumptions were made in \cite{causemann2022human}, although we do not make use of the Beavers-Joseph-Saffman condition 
\cite{mikelic2000interface}.

Aiming at solving problem \eqref{eq:NSMPE} with the Finite Element method, we introduce its weak formulation.
For the sake of generality, let $\Gamma_{\text{D},\vec{d}},\Gamma_{\text{D},\vec{u}},\Gamma_{\text{D},P_\jj}$, with $\jj\in J$, denote the portions of $\partial\Omega$ where Dirichlet boundary conditions on $\vec{d},\vec{u},p_\jj$ are imposed, respectively.
Then, we introduce the following functional spaces:
\[\begin{gathered}
\vec{W} = \{\vec{w}\in[H^1(\Omega_\PP)]^d \colon \vec{w}=0 \text{ on }\Gamma_{\text{D},\vec{d}}\},\qquad
\vec{V} = \{\vec{v}\in[H^1(\Omega_\FF)]^d\colon \vec{v}=0 \text{ on }\Gamma_{\text{D},\vec{u}}\},\\
Q_\jj = \{q_\jj\in H^1(\Omega_\PP)\colon q_\jj=0 \text{ on }\Gamma_{\text{D},P_\jj}\}, \ \forall\jj\in J,\qquad
Q = L^2(\Omega_\FF),
\end{gathered}\]
where $H^1(\Omega)$ denotes the classical Sobolev space of order 1 over $L^2(\Omega)$.
For the problem at hand, all Dirichlet boundary conditions on $\partial\Omega$ are homogenous.
We denote by $(\cdot,\cdot)_{\Omega}$ the $L^2$-product over $\Omega$ and we define the following forms and functionals over the spaces introduced above:

\[\begin{aligned}
a_\PP:\vec{W}\times\vec{W}\to\mathbb R,
&\qquad a_\PP(\vec{d},\vec{w}) = (\sigma_\PP(\vec{d}),\varepsilon(\vec{w}))_{\Omega_\PP},\\
a_\jj:Q_\jj\times Q_\jj\to\mathbb R,
&\qquad a_\jj(p_\jj,q_\jj) = \left(\frac{1}{\mu_\jj}K_\jj\nabla p_\jj,\nabla q_\jj\right)_{\Omega_\PP} \qquad\forall\jj\in J,\\
C_\jj:\left(\bigtimes_{\kk\in J}Q_\kk\right)\times Q_\jj\to\mathbb R,
&\qquad
C_\jj(\{p_\kk\}_{\kk\in J},q_\jj) = 
\sum_{\kk\in J}(\beta_{\kk\jj}(p_\jj-p_\kk), q_\jj)_{\Omega_\PP}
+ (\beta_{\jj}^\text{e}p_\jj, q_\jj)_{\Omega_\PP} \qquad\forall\jj\in J,
\\
a_\FF:\vec{V}\times\vec{V}\to\mathbb R,
&\qquad a_\FF(\vec{u},\vec{v}) = (\sigma_\FF(\vec{u}),\varepsilon(\vec{v}))_{\Omega_\FF},\\
b_\jj:Q_\jj\times\vec{W}\to\mathbb R,
&\qquad b_\jj(q_\jj,\vec{w}) = -(\alpha_\jj p_\jj,\Div\vec{w})_{\Omega_\PP} \qquad\forall\jj\in J,\\
b_\FF:Q\times\vec{V}\to\mathbb R,
&\qquad b_\FF(q,\vec{v}) = -(q,\Div\vec{v})_{\Omega_\FF},\\
F_\PP:\vec{W}\to\mathbb R,
&\qquad F_\PP(\vec{w}) = (\vec{f}_\PP,\vec{w})_{\Omega_\PP},\\
F_\jj:Q_\jj\to\mathbb R,
&\qquad F_\jj(q_\jj) = (g_\jj,q_\jj)_{\Omega_\PP} \qquad\forall\jj\in J,\\
F_\FF:\vec{V}\to\mathbb R,
&\qquad F_\FF(\vec{v}) = (\vec{f}_\FF,\vec{v})_{\Omega_\FF},\\
\mathfrak J:Q_\EE\times\vec{W}\times\vec{V}\to\mathbb R,
&\qquad \mathfrak J(p_\EE,\vec{w},\vec{v}) = \int_\Sigma p_\EE\left(\vec{w}\cdot\vec{n}_\PP+\vec{v}\cdot\vec{n}_\FF\right)d\Sigma.
\end{aligned}\]

\begin{remark}[Derivation of the interface form $\mathfrak J$]\label{rem:J}
The interface form $\mathfrak J:\spaceQE\times\spaceW\times\spaceV\to\mathbb R$ introduced 
above 
naturally arises during the derivation of the weak form of problem \eqref{eq:NSMPE}.
We test \eqref{eq:elasticity}-\eqref{eq:pj} against functions $\vec{w}\in\vec{W}$ and $q_\jj\in Q_\jj$, with $\jj\in J$, over $\Omega_\PP$, and \eqref{eq:fluidMom} against $\vec{v}\in \vec{V}$ over $\Omega_\FF$. Then, integrating by parts and summing all the contributions yield the following boundary terms on the interface:
\begin{equation}\label{eq:neumanntermsCONT}
    \int_\Sigma\left[
    (pI-\sigma_\FF(\vec{u}))\colon \vec{v}\otimes\vec{n}_\FF
    + \left(\sum_{\kk\in J}\alpha_\kk p_\kk I-\sigma_\PP(\vec{d})\right)\colon \vec{w}\otimes\vec{n}_\PP
    - \sum_{\jj\in J}\frac{1}{\mu_\jj}K_\jj\nabla p_\jj\cdot q_\jj\vec{n}_\PP
    \right]d\Sigma.
\end{equation}
Using the interface conditions \eqref{eq:BCtotalstress}-\eqref{eq:BCpnonE} and then \eqref{eq:BCnormalflux}-\eqref{eq:BCnormalstress}-\eqref{eq:BCtgstress}, we can rewrite \eqref{eq:neumanntermsCONT} as follows:
\begin{equation}
\begin{aligned}
\int_\Sigma&\left[
(pI-\sigma_\FF(\vec{u}))\colon (\vec{v}\otimes\vec{n}_\FF + \vec{w}\otimes\vec{n}_\PP) - \frac{1}{\mu_\EE}K_\EE\nabla p_\EE\cdot q_\EE\vec{n}_\PP
\right]d\Sigma \\
&= \int_\Sigma\left[
p_\EE (\vec{v}\cdot\vec{n}_\FF + \vec{w}\cdot\vec{n}_\PP) - q_\EE(\vec{u}\cdot\vec{n}_\FF + \partial_t\vec{d}\cdot\vec{n}_\PP)
\right]d\Sigma
=\mathfrak J(p_\EE,\vec{w},\vec{v}) - \mathfrak J(q_\EE,\partial_t\vec{d},\vec{u}),
\end{aligned}\end{equation}
where we also used that $\vec{a}\otimes\vec{b}\colon I=\vec{a}\cdot\vec{b}$ for any $\vec{a},\vec{b}\in \mathbb R^d$.
\end{remark}

Denoting by $L^2(0,T;H), H^1(0,T; H)$ the time-dependent Bochner spaces associated to a Sobolev space $H$, 
and setting
\[
\mathscr{D}=H^2(0,T; \vec{W}),\ \ \mathscr{P}=\bigtimes_{\jj\in J}H^1(0,T; Q_\jj),\ \ \mathscr{V}=H^1(0,T; \vec{V}),\ \ \mathscr{Q}=L^2(0,T;Q),
\]
the weak formulation of problem \eqref{eq:NSMPE} reads as follows:\\
Find $(\vec{d},\{p_\jj\}_{\jj\in J},\vec{u},p)\in \mathscr{D}\times\mathscr{P}\times\mathscr{V}\times\mathscr{Q}$ such that, for all $t\in(0,T]$,
\begin{equation}\label{eq:weak}\begin{aligned}
    &(\rho_\PP\partial_{tt}^2\vec{d},\vec{w})_{\Omega_\PP} + a_\PP(\vec{d},\vec{w}) + \sum_{\jj\in J} b_\jj(p_\jj,\vec{w}) - F_\PP(\vec{w}) \\
    &\qquad+ \sum_{\jj\in J}\Bigl[(c_\jj\partial_t p_\jj,q_\jj)_{\Omega_\PP} + a_\jj(p_\jj,q_\jj) + C_j(\{p_\kk\}_{\kk\in J},q_\jj) - b_\jj(q_\jj,\partial_t\vec{d}) - F_\jj(q_\jj) \Bigr]\\
    &\qquad+ (\rho_\FF\partial_t\vec{u},\vec{v})_{\Omega_\FF} + a_\FF(\vec{u},\vec{v}) + b_\FF(p,\vec{v})+b_\FF(q,\vec{u}) - F_\FF(\vec{v})\\
    &\qquad+ \mathfrak J(p_\EE,\vec{w},\vec{v}) - \mathfrak J(q_\EE,\partial_t\vec{d},\vec{u}) = 0\\
\end{aligned}\end{equation}
for all $(\vec{w},\{q_\jj\}_{\jj\in J},\vec{v},q) \in \mathscr{D}\times\mathscr{P}\times\mathscr{V}\times\mathscr{Q}$,
and $\vec{d}(0)=\vec{d}_0, \partial_t\vec{d}(0)=\dot{\vec{d}}_0, \vec{u}(0)=\vec{u}_0, p_\jj(0)=p_{\jj0}$ $\forall \jj\in J$.

\section{Semidiscrete formulation based on a polytopal discontinuous Galerkin method}\label{sec:polydg}
In this section, we introduce a space discretization of problem \eqref{eq:weak} based on discontinuous Finite Element methods on polytopal grids.

\subsection{Notation}\label{sec:notation}
Let $\mathscr T_{h,\PP},\mathscr T_{h,\FF}$ be polytopal meshes discretizing the domains $\Omega_\PP,\Omega_\FF$, respectively.
We define as \emph{faces} of an element $K\in \mathscr T_{h,\PP}\cup\mathscr T_{h,\FF}$ the $(d-1)$-dimensional entities corresponding to the intersection of $\partial K$ with either the boundary of a neighboring element or the domain boundary $\partial\Omega$:
\begin{itemize}
    \item for $d=2$, the faces are always straight line segments;
    \item for $d=3$, the faces are generic polygons. We assume that each face can be decomposed into triangles.
\end{itemize}
With this definition, we denote by $\mathscr F_\PP, \mathscr F_\FF$ the sets of element faces corresponding to each physical domain.
We partition them into internal faces $\mathscr F_\PP^\II,\mathscr F_\PP^\II$, Dirichlet/Neumann faces $\mathscr F_\PP^\DD/\mathscr F_\PP^\NN\subset\partial\Omega_\PP\setminus\Sigma,\mathscr F_\PP^{\DD_\jj}/F_\PP^{\NN_\jj}\subset\partial\Omega_\PP\setminus\Sigma,\mathscr F_\FF^\DD/\mathscr F_\FF^\NN\subset\partial\Omega_\FF\setminus\Sigma$ (for the elastic displacement, the pressure of the $\jj$-th compartment, and the fluid velocity, respectively), and interface faces $\mathscr F^\Sigma\subset\Sigma$.
In the latter, we assume that the meshes $\mathscr T_{h,\PP},\mathscr T_{h,\FF}$ are aligned with $\Sigma$, namely that there is no gap or overlap between them, although hanging nodes are permitted. 

We introduce the symmetric outer product $\vec{v}\odot\vec{n} = \frac{1}{2}(\vec{v}\otimes\vec{n}+\vec{n}\otimes\vec{v})$ and, for regular enough scalar-, vector- and tensor-valued functions $q,\vec{v},\tau$,
we define the following average and jump operators:
\begin{itemize}
    \item On each internal face $F\in\mathscr F^\II=\mathscr F_\PP^\II\cup \mathscr F_\FF^\II$ we set
    \begin{align*}
    \average{q} &= \frac{1}{2}(q^++q^-),
    &\average{\vec{v}} &= \frac{1}{2}(\vec{v}^++\vec{v}^-),
    &\average{\tau} &= \frac{1}{2}(\tau^++\tau^-),
        \\
    \jump{q} &= q^+\vec{n}^++q^-\vec{n}^-,
    &\jump{\vec{v}} &= \vec{v}^+\odot\vec{n}^++\vec{v}^-\odot\vec{n}^-,
    &\jump{\tau} &= \tau^+\vec{n}^++\tau^-\vec{n}^-.
    \end{align*}
    where $\vec{n}^+,\vec{n}^-$ are defined as in \cref{fig:FSigma} - left.
    \item On a Dirichlet face $F\in\mathscr F_\PP^\DD\cup\left(\bigcup_{\jj\in J}\in\mathscr F_\PP^{\DD_\jj}\right)\cup\mathscr F_\FF^\DD$:
    \begin{align*}
    \average{q} &= q,
    &\average{\vec{v}} &= \vec{v},
    &\average{\tau} &= \tau,
        \\
    \jump{q} &= q\vec{n},
    &\jump{\vec{v}} &= \vec{v}\odot\vec{n},
    &\jump{\tau} &= \tau\vec{n},
    \end{align*}
    where $\vec{n}$ is the unit normal vector pointing outward to the element $K$ to which the face $F$ belongs.
    \item On a face $F\in\mathscr F^\Sigma$ shared by two elements $K_\PP\in\mathscr T_{h,\PP}$ and $K_\FF\in\mathscr T_{h,\FF}$:
    \begin{align*}
        \average{q} &= q|_{K_\PP},
        \quad&\average{\tau} &= \tau|_{K_\PP},
        \quad&\jump{\vec{w},\vec{v}} &= \vec{w}|_{K_\PP}\odot\vec{n}_\PP+\vec{v}|_{K_\FF}\odot\vec{n}_\FF,
    \end{align*}
    where $\vec{n}_\PP,\vec{n}_\FF$ are defined as in \cref{fig:FSigma} - right.
\end{itemize}

\begin{figure}
    \centering
    \includegraphics[width=0.27\textwidth]{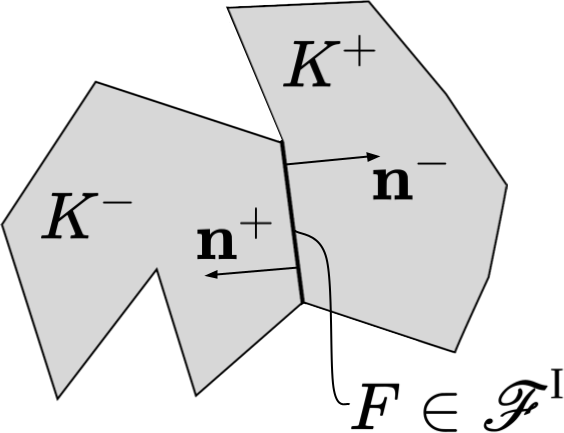}
    \qquad\qquad
    \includegraphics[width=0.3\textwidth]{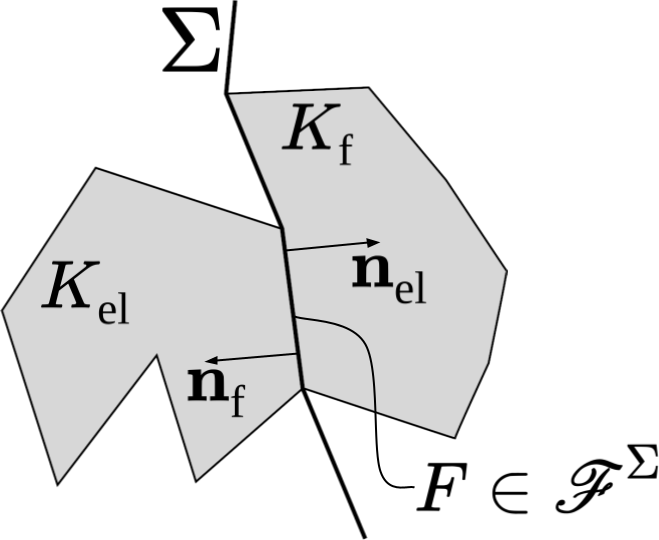}
    \caption{Polygonal elements sharing an internal face (left) or a face on the interface $\Sigma$ (right).}
    \label{fig:FSigma}
\end{figure}

\subsection{PolyDG semidiscrete problem}\label{sec:polydgsub}

For a given integer $m\geq1$, we introduce the following piecewise polynomial spaces:
\begin{gather*}
    X_h^\text{DG}(\Omega_\star) =\{\phi\in L^2(\Omega_\star)\colon\phi|_K\in\mathbb P^m(K)\quad\forall K\in\mathscr T_\star\}, \qquad \star=\PP,\FF\\
    \spaceQj = X_h^\text{DG}(\Omega_\PP),
    \quad
    \spaceQ = X_h^\text{DG}(\Omega_\FF),
    \quad
    \spaceW = [X_h^\text{DG}(\Omega_\PP)]^d,
    \quad
    \spaceV = [X_h^\text{DG}(\Omega_\FF)]^d.
\end{gather*}

Moreover, we denote by $H^s(\mathscr T_\PP),H^s(\mathscr T_\FF)$ the broken Sobolev spaces of order $s$ over the mesh of the poroelastic and fluid domains, namely
$H^s(\mathscr T_\star) = \{q\in L^2(\Omega_\star) \colon q|_K\in H^s(K)\quad\forall K\in\mathscr T_\star\}$
for $\star=\PP,\FF$.
We then introduce the following forms, for $\vec{d},\vec{w}\in[H^1(\mathscr T_{h,\PP})]^d, p_\jj,q_\jj\in H^1(\mathscr T_{h,\PP}), \text{ with }\jj\in J, \vec{u},\vec{v}\in[H^1(\mathscr T_{h,\FF})]^d, p,q\in H^1(\mathscr T_{h,\FF})$:

\begin{subequations}\label{eq:formsSeparated}\begin{align}
    \begin{split}
    \mathcal A_\PP(\vec{d},\vec{w}) &= \int_{\Omega_\PP}\sigma_\PP(\vec{d})\colon\varepsilon_h(\vec{w})
    \\&\qquad
    - \sum_{F\in\mathscr F_\PP^\II\cup\mathscr F_\PP^\DD}
    \int_F\left(\average{\sigma_\PP(\vec{d})}\colon\jump{\vec{w}}+\jump{\vec{d}}\colon\average{\sigma_\PP(\vec{w})}
    -
    \eta\jump{\vec{d}}\colon\jump{\vec{w}}\right),
    \end{split}\\
    \mathcal F_\PP(\vec{w}) &= \int_{\Omega_\PP}\vec{f}_\PP\cdot\vec{w}
    ,\\
    \mathcal B_\jj(p_\jj,\vec{w}) &= -\int_{\Omega_\PP}\alpha_\jj p_\jj\,\Div_h\,\vec{w}
    + \sum_{F\in\mathscr F_\PP^\II\cup\mathscr F_\PP^{\DD_\jj}}
    \int_F\alpha_\jj \average{p_\jj I}\colon\jump{\vec{w}}, 
    ,\\
    \begin{split}\mathcal A_\jj(p_\jj,q_\jj) &= \int_{\Omega_\PP}\mu_\jj^{-1}K_\jj\nabla_h p_\jj\cdot\nabla_h q_\jj
    \\&\hspace{-1em}- \sum_{F\in\mathscr F_\PP^\II\cup\mathscr F_\PP^{\DD_\jj}}
    \int_F\left(\average{\mu_\jj^{-1}K_\jj\nabla_h p_\jj }\cdot\jump{q_\jj } +\jump{p_\jj }\cdot\average{\mu_\jj^{-1}K_\jj\nabla_h q_\jj }
    -
    \zeta_\jj\jump{p_\jj}\cdot\jump{q_\jj}
    \right)
    ,\end{split}\\
    \mathcal C_\jj(\{p_\kk\}_{\kk\in J},q_\jj) &= \int_{\Omega_\PP}\sum_{\kk\in J}\beta_{\kk\jj}(p_\jj-p_\kk)q_\jj + \int_{\Omega_\PP}\beta_\jj^\text{e}p_\jj q_\jj
    ,\\
    \mathcal F_\jj(q_\jj) &= \int_{\Omega_\PP}g_\jj q_\jj 
    ,\\
    \begin{split}\mathcal A_\FF(\vec{u},\vec{v}) &= \int_{\Omega_\FF}\sigma_\FF(\vec{u})\colon\varepsilon_h(\vec{v})
    \\&- \sum_{F\in\mathscr F_\FF^\II\cup\mathscr F_\FF^\DD}
    \int_F\left(\average{\sigma_\FF(\vec{u})}\colon\jump{\vec{v}}+\jump{\vec{u}}\colon\average{\sigma_\FF(\vec{v})}
    -
    \gamma_\vec{v}\jump{\vec{u}}\colon\jump{\vec{v}}
    \right)
    ,\end{split}\\
    \mathcal B_\FF(p,\vec{v}) &= -\int_{\Omega_\FF} p\,\Div_h\,\vec{v}
    + \sum_{F\in\mathscr F_\FF^\II\cup\mathscr F_\FF^{\DD}}
    \int_F \average{p I}\colon\jump{\vec{v}}
    ,\\
    \mathcal F_\FF(\vec{v}) &= \int_{\Omega_\FF}\vec{f}_\FF\cdot\vec{v}
    ,\\
    \mathcal S(p,q) &=
    \sum_{F\in\mathscr F_\FF^\II}
    \int_F\gamma_p\jump{p}\cdot\jump{q}
    ,\\
    \mathcal J(p_\EE, \vec{w},\vec{v}) &=
        \sum_{F\in\mathscr{F}^\Sigma}\int_F\left(
    \average{p_\EE I} \colon \jump{\vec{w},\vec{v}}
    \right),
\end{align}\end{subequations}
where $\nabla_h, \varepsilon_h, \Div_h$ denote the element-wise gradient, symmetric gradient, and divergence operators, respectively, and the stress tensors $\sigma_\PP,\sigma_\FF$ are implicitly defined in terms of these piecewise operators.
The parameters $\eta,\zeta_\jj,\gamma_\vec{v},\gamma_p$ appearing in these forms are defined as follows \cite{AMVZ22,corti2022numerical}:
\begin{equation}\label{eq:penaltyparams}
\eta = \overline{\eta}\frac{\overline{\mathbb C}_\PP^K}{\{h\}_\text{H}},
\qquad
\zeta_\jj = \overline{\zeta}_\jj\frac{\overline{k}_\jj^K}{\sqrt{\mu_\jj}\{h\}_\text{H}},
\qquad
\gamma_\vec{v} = \overline{\gamma}_\vec{v}\frac{\mu}{\{h\}_\text{H}},
\qquad
\gamma_p = \overline{\gamma}_p\{h\}_\text{H},
\end{equation}
where
$\{h\}_\text{H}$ denotes the harmonic average on $K^{\pm}$ (with $\{h\}_\text{H}=h_K$ on Dirichlet faces)
$\overline{\mathbb C}_\PP^K = \|\mathbb C_\PP^{1/2}|_K\|_2^2, \overline{k}_\jj^K = \|K_\jj^{1/2}|_K\|_2^2$ are the $L^2$-norms of the symmetric second-order tensors appearing in the elasticity and Darcy equations, for each $K\in\mathscr T_\PP$,
and
$\overline{\eta},\overline{\zeta_\jj}\ \forall\jj\in J,\overline{\gamma}_\vec{v},\overline{\gamma}_p$ are penalty constants to be chosen large enough.

The form $J(\cdot,\cdot,\cdot)$ is the piecewise discontinuous version of the interface form $\mathfrak J(\cdot,\cdot,\cdot)$ derived in \cref{rem:J}.
Indeed, for any $p_\EE\in H^1(\mathscr T_{h_\PP}),\vec{w}\in [H^1(\mathscr T_{h,\PP}]^d,\vec{v}\in[H^1(\mathscr T_{h,\FF})]^d$,
\begin{equation}\label{eq:Jdiscr}\begin{aligned}
\mathfrak J(q_\EE,\vec{w},\vec{v})
&=
\sum_{F\in\mathscr{F}^\Sigma}\int_F\left[
    p_\EE I \colon (\vec{w}\odot\vec{n}_\PP + \vec{v}\odot\vec{n}_\FF)
    - q_\EE I \colon(\partial_t\vec{d}\odot\vec{n}_\PP + \vec{u}\odot\vec{n}_\FF)
    \right]\\
&=
    \sum_{F\in\mathscr{F}^\Sigma}\int_F\left[
    \average{p_\EE I} \colon \jump{\vec{w},\vec{v}} - \jump{\partial_t\vec{d},\vec{u}} \colon \average{q_\EE I}
    \right],
\end{aligned}\end{equation}
where we have used the identity $\vec{a}\cdot\vec{b} = {\vec{a}\odot\vec{b}\ \colon I}, \forall\vec{a},\vec{b}\in\mathbb R^d$ and then the definition of the average and jump operators introduced in \cref{sec:notation}.

Finally, the semidiscrete formulation reads as follows:
\\
    For any $t\in(0,T]$, find $(\vec{d}_h,\{p_{\jj,h}\}_{\jj\in J},\vec{u}_h,p_h)\in \spaceW\times\left(\bigtimes_{\jj\in J} \spaceQj \right) \times \spaceV \times \spaceQ$ such that
\begin{equation}\label{eq:DG}\begin{aligned}
    (\rho_\PP\partial_{tt}^2\vec{d}_h&,\vec{w}_h)_{\Omega_\PP} + \mathcal L_\PP(\vec{d}_h,\{p_{\kk,h}\}_{\kk\in J}; \vec{w}_h) - \mathcal F_\PP(\vec{w}_h) \\
    & + \sum_{\jj\in J} \left[ (c_\jj\partial_t p_{\jj,h}, q_{\jj,h})_{\Omega_\PP} + \mathcal L_\jj(\{p_{\kk,h}\}_{\kk\in J},\partial_t\vec{d}_h;q_{\jj,h}) - \mathcal F_\jj(q_{\jj,h}) \right] \\
    & + (\rho_\FF\partial_t\vec{u}_h,\vec{v}_h)_{\Omega_\FF}+\mathcal L_\FF(\vec{u}_h,p_h;\vec{v}_h,q_h) - \mathcal F_\FF(\vec{v}_h)
    \\
    & + \mathcal J(p_{\EE,h}, \vec{w}_h,\vec{v}_h) - \mathcal J (q_{\EE,h}, \partial_t\vec{d}_h,\vec{u}_h) = 0 \\ &\forall\vec{w}_h\in\spaceW,\vec{v}_h\in\spaceV,q_h\in\spaceQ,q_{\jj,h}\in\spaceQj.
\end{aligned}\end{equation}
Problem \eqref{eq:DG} is supplemented with suitable initial conditions $\vec{d}_h(0), \dot{\vec{d}}_h(0), \{p_{\jj,h}(0)\}_{\jj\in J}, \vec{u}_h(0)$ that are projections of the initial data introduced in \eqref{eq:NSMPE} onto the corresponding DG spaces.
The bilinear forms appearing in \eqref{eq:DG} are defined as
\begin{subequations}\label{eq:formsAllTogether}\begin{align}
    \mathcal L_\PP(\vec{d},\{p_\kk\}_{\kk\in J}; \vec{w}) &= \mathcal A_\PP(\vec{d},\vec{w}) + \sum_{\kk\in J}\mathcal B_\kk(p_\kk,\vec{w}) 
    , \\
    \mathcal L_\jj(\{p_\kk\}_{\kk\in J},\partial_t\vec{d};q_\jj) &=  \mathcal A_\jj(p_\jj,q_\jj) + \mathcal C_\jj(\{p_\kk\}_{\kk\in J},q_\jj) - \mathcal B_\jj(q_\jj,\partial_t \vec{d}) 
    , \qquad \forall\jj\in J, \\
    \mathcal L_\FF(\vec{u},p;\vec{v},q) &= \mathcal A_\FF(\vec{u},\vec{v}) + \mathcal B_\FF(p,\vec{v}) 
    - \mathcal B_\FF(q,\vec{u}) + \mathcal S(p,q).
\end{align}\end{subequations}

\subsection{Algebraic formulation}\label{sec:algSemiDiscr}
We introduce suitable sets of basis functions such that 
$
{\rm span}\{\bm\varphi_\PP^i\}_{i=0}^{N_\PP}=\spaceW$, ${\rm span}\{\bm\varphi_\PP^i\}_{i=0}^{N_\FF}=\spaceV$, ${\rm span}\{\psi^i\}_{i=0}^{N_p}=\spaceQ$, ${\rm span}\{\psi_\jj^i\}_{i=0}^{N_\jj}=\spaceQj$ for $\jj\in J$.
Denoting by uppercase letters the d.o.f.~vectors corresponding to the problem unknowns, the (formal) algebraic form of \eqref{eq:DG} is the following:\\
Given $\vec{D}_0,\dot{\vec{D}}_0,\vec{U}_0,\vec{P}_{\jj0},\jj\in J$, find $\vec{D},\vec{U},\vec{P},\vec{P}_\jj,\jj\in J$ such that
\begin{equation}\label{eq:formalalg}
\begin{bmatrix}
M_\PP\partial_{tt}^2 + A_\PP & B_\text{A}^T & \cdots & B_\text{E}^T+J_{\PP}^T & 0 & 0 \\
-B_\text{A}\partial_t & M_\text{A}\partial_t + A_\text{A} + C_{\text{A}\text{A}} & \cdots & C_{\text{A}\text{E}} & 0 & 0\\
\vdots &\vdots &\vdots &\vdots &\vdots &\vdots\\
-B_\EE + J_{\PP}\partial_t & C_{\EE\text{A}} & \cdots & M_\EE\partial_t + A_\EE + C_{\EE\EE} & -J_{\FF} & 0 \\
0 & 0 & \cdots & J_{\FF}^T & M_\FF\partial_t + A_\FF & B_\FF^T \\
0 & 0 & \cdots & 0 & -B_\FF & S
\end{bmatrix}
\begin{bmatrix}
\vec{D} \\ \vec{P}_\text{A} \\ \vdots \\ \vec{P}_\EE \\ \vec{U} \\ \vec{P}
\end{bmatrix}
=
\begin{bmatrix}
\vec{F}_\PP \\ \vec{F}_\text{A} \\ \vdots \\ \vec{F}_\EE \\ \vec{F}_\FF \\ \vec{0}
\end{bmatrix}
\end{equation}
The matrices and vectors of \eqref{eq:formalalg} are defined as follows
(where $\star=\PP,\FF$):
\begin{align*}
[M_\star]_{ij} &= (\rho_\star\bm\varphi_\star^j,\bm\varphi_\star^i)_{\Omega_\star}, &
[M_\jj]_{ij} &= (c_\jj\psi_\jj^j,\phi_\jj^i)_{\Omega_\PP}, &
[S]_{ij} &= \mathcal S(\psi^j,\phi^i)
\\
[A_\star]_{ij} &= \mathcal A_\star(\bm\varphi_\star^j,\bm\varphi_\star^i), &
[A_\jj]_{ij} &= \mathcal A_\jj(\psi_\jj^j,\phi_\jj^i), &
[C_\jj]_{ij} &= \mathcal C_\jj(\{\psi_\kk^i\}_{\kk\in J},\psi_\jj^i), \\
[B_\FF]_{ij} &= \mathcal B_\FF(\psi^i,\bm\varphi_\FF^j), &
[B_\jj]_{ij} &= \mathcal B_\jj(\psi_\jj^i,\varphi_\PP^i), &
[J_\star]_{ij} &= \sum_{F\in\mathscr F^\Sigma}\int_F \average{q_\text{E}^j I}\colon \bm\varphi_\star^i\odot\vec{n}_\star, \\
[F_\star]_i &= \mathcal F_\star(\vec{\varphi}_\star^i), &
[F_\jj]_i &= \mathcal F_\jj(\psi_\jj^i).
\end{align*}

\section{A priori analysis of the semidiscrete problem}\label{sec:apriori}

For the analysis contained in this section, we consider a generic set $J$ of $N_J\in\mathbb N_0$ compartments and we assume that all the physical parameters of the model (defined in \cref{tab:modelparams}) are piecewise constant.
For $r\geq 1$, we introduce the following broken norms \cite{corti2022numerical,AMVZ22}:
\begin{subequations}\label{eq:brokennorms}\begin{align}
    \normDGd{\vec{d}}^2 &= \|{\mathbb C_\PP^{1}{2}}[\varepsilon_h(\vec{d})]\|_{L^2(\mathscr T_\PP)}^2 + \|\sqrt{\eta}\jump{\vec{d}}\|_{\mathscr F_{\PP,h}^{\II}\cup\mathscr F_{\PP,h}^{\DD}}^2
    &\forall\vec{d}\in \vec{H}^r(\mathscr T_h),
    \\
    \normDGj{p}^2 &= \|\mu_\jj^{-1/2}K_\jj^{1/2}\nabla_h p\|_{L^2(\mathscr T_\PP)}^2 + \|\sqrt{\zeta_\jj}\jump{p}\|^2_{\mathscr F_{\PP,h}^{\II}\cup\mathscr F_{\PP,h}^{\DD_\jj}}
    &\forall p \in H^r(\mathscr T_h),
    \\
    \normDGu{\vec{u}}^2 &= \|\sqrt{2\mu}\,\varepsilon_h(\vec{u})\|_{L^2(\mathscr T_\FF)}^2 + \|\sqrt{\gamma_\vec{v}}\jump{\vec{u}}\|^2_{\mathscr F_{\FF,h}^{\II}}
    &\forall\vec{u}\in \vec{H}^r(\mathscr T_h),
    \\
    \normDGp{q}^2 &= \|q\|_{L^2(\Omega_\FF)}^2 + \|\sqrt{\gamma_p}\jump{q}\|^2_{\mathscr F_{\FF,h}^{\II}\cup\mathscr F_{\FF,h}^{\DD}}
    &\forall q\in H^r(\mathscr T_h).
\end{align}\end{subequations}
We introduce the energy norms at time $t\in(0,T]$
\begin{align*}
    \normENporoel{(\vec{d},\{p_\jj\}_{\jj\in J})}
    &= \left[\|\sqrt{\rho_\PP}\partial_t\vec{d}(t)\|_{\Omega_\PP}^2 + \normDGd{\vec{d}(t)}^2 
    \phantom{\sum_{\jj\in J}\int_0^t}\right.\\&\qquad\left.
    + \sum_{\jj\in J}\left(\|\sqrt{c_\jj}p_\jj(t)\|_{\Omega_\PP}^2 + \int_0^t\left(\normDGj{p_\jj(s)}^2+\|\sqrt{\beta^\text{e}_\jj}p_\jj(s)\|_{\Omega_\PP}^2\right)ds\right)\right]^{1/2},
    \\
    \normENfluid{(\vec{u},p)}
    &= \left[\|\sqrt{\rho_\FF}\vec{u}(t)\|_{\Omega_\FF}^2 + \int_0^t\left(\normDGu{\vec{u}(s)}^2 +\normDGp{p(s)}^2\right)ds\right]^{1/2},
\end{align*}
and set
\begin{align}\label{eq:normcoerc}
    \normEN{(\vec{d},\{p_\jj\}_{\jj\in J},\vec{u},p)}
    &= \left[ \normENporoel{(\vec{d},\{p_\jj\}_{\jj\in J})}^2+ \normENfluid{(\vec{u},p)}^2\right]^{1/2}.
\end{align}

For the sake of simplicity, in the inequalities appearing in the following analysis, we neglect the dependencies on the model parameters and the polynomial degree $m$, and we use the notation $x\lesssim y$ to indicate that $\exists C>0:x\leq C y$, where $C$ is independent of the mesh discretization parameters.
Moreover, we make the following assumption on the mesh \cite{antonietti2016review,cangiani2017book}:
\begin{hp}\label{hp:mesh}
For each $h>0$, the two meshes $\mathscr T_{h,\PP},\mathscr T_{h,\FF}$ are aligned with $\Sigma$, namely there is no gap nor overlap between them (hanging nodes are allowed).
Moreover, denoting by $\mathscr T_h$ the union of $\mathscr T_{h,\PP}$ and $\mathscr T_{h,\FF}$, we consider a sequence of meshes $\{\mathscr T_h\}_{h>0}$ satisfying the regularity requirements of \cite
{corti2022numerical}:
\begin{itemize}
    \item
    $\{\mathscr T_h\}_{h>0}$ is $h$-uniformly polytopic-regular, namely for each $K\in\mathscr T_h$ there exists a set $\{S_K^F\}_{F\subset\partial K}$ of non-overlapping $d$-dimensional simplices contained in $\overline{K}$ such that
    \[
    \forall F\subset \partial K \colon \overline{F} = \overline{\partial K}\cap\overline{S_K^F} \text{ it holds that } h_K\lesssim d|S_K^F||F|^{-1},
    \]
    where $h_K$ is the diameter of $K$, and the cardinality of $\{S_K^F\}_{F\subset\partial K}$ does not depend on $h>0$.
    \item
    There exists a shape-regular simplicial covering $\widehat{\mathscr T}_h$ of $\mathscr T_h$ such that, for each $K\in\mathscr T_h, \widehat{K}\in\widehat{\mathscr T}_h$ with $\widehat{K}\subseteq K$,
    \[
    h_K\lesssim h_{\widehat{K}},
    \max_{K\in\mathscr T_h}{\rm card}\{K'\in\mathscr T_h \colon \exists \widehat{K}\in\widehat{\mathscr T}_h \text{ such that } K\subset\widehat{K} \text{ and } K'\cap\widehat{K}\neq\emptyset\}\lesssim 1.
    \]
    \item
    The mesh size satisfies a local bounded variation property:
    \[
    \forall K_1,K_2\in\mathscr T_h\colon|\partial K_1\cap\partial K_2|_{d-1}>0 \qquad h_{K_1}\lesssim h_{K_2} \lesssim h_{K_1},
    \]
    with $|\cdot|_{d-1}$ denotes the $(d-1)$-dimensional measure and all the hidden constants independent of $h$ and the number of faces of $K_1$ and $K_2$.
\end{itemize}
\end{hp}

Under these assumptions, we can prove the following stability result:
\begin{theorem}[Stability estimate]\label{th:stab}
Let \cref{hp:mesh} hold true and let us also assume that 
the penalty constants defined in \eqref{eq:penaltyparams} are chosen sufficiently large.
Then, the semidiscrete solution $(\vec{d}_h,\{p_{\jj,h}\}_{\jj\in J},\vec{u}_h,p_h)$ of \eqref{eq:DG} satisfies the following inequality for each $t\in(0,T]$:
\begin{equation}\label{eq:stab}\begin{aligned}
    \normEN{(\vec{d}_h,\{p_{\jj,h}\}_{\jj\in J},\vec{u}_h,p_h)} \lesssim
    & \normENzero{(\vec{d}_h,\{p_{\jj,h}\}_{\jj\in J},\vec{u}_h,0)}
    \\& + \int_0^t\left(\frac{1}{\sqrt{\rho_\PP}}\|\vec{f}_\PP\|_{\Omega_\PP}+\sum_{\jj\in J}\frac{1}{\sqrt{c_\jj}}\|g_\jj\|_{\Omega_\PP}+\frac{1}{\sqrt{\rho_\FF}}\|\vec{f}_\FF\|_{\Omega_\FF}\right)ds,
\end{aligned}\end{equation}
where, according to the initial conditions of \eqref{eq:NSMPE},
\[
\normENzero{(\vec{d}_h,\{p_{\jj,h}\}_{\jj\in J},\vec{u}_h,0)} = \left[\|\sqrt{\rho_\PP}\dot{\vec{d}}_h^0\|_{\Omega_\PP}^2 + \normDGd{\vec{d}_h^0}^2 + \sum_{\jj\in J}\|\sqrt{c_\jj}p_{\jj,h}^0\|_{\Omega_\PP}^2 + \|\sqrt{\rho_\FF}\vec{u}_h^0\|_{\Omega_\FF}^2\right]^{1/2}.
\]
\end{theorem}
\begin{proof}
Let us fix a time $t\in (0,T]$ and consider the following test functions in \eqref{eq:weak}: $\vec{w}=\partial_t\vec{d}(t), \vec{v}=\vec{u}(t), q=p(t), q_\jj=p_\jj(t)\ \forall\jj\in J$.
With this choice of test functions, the terms $\mathcal J$ cancel out, as well as the $\mathcal B$ forms in the $\mathcal L$ terms
\footnote{In the next steps, although the $\mathcal B_\FF$ terms cancel out due to the choice of test functions, we keep writing the complete form $\mathcal L_\FF$ to facilitate the application of a generalized inf-sup condition (see \cref{sec:contcoerc}, \cref{th:contcoerc}).}
.
Therefore, \eqref{eq:weak} becomes
\begin{equation}\label{eq:proofstab1}\begin{aligned}
(\rho_\PP\partial_{tt}^2&\vec{d}(t),\partial_t\vec{d}(t))_{\Omega_\PP} + \mathcal A_\PP(\vec{d}(t),\partial_t\vec{d}(t))
\\&\quad
+ \sum_{\jj\in J}\left[ (c_\jj\partial_t p_\jj(t),p_\jj(t))_{\Omega_\PP} + \mathcal A_\jj(p_\jj(t),p_\jj(t)) +\mathcal C_\jj(\{p_\kk(t)\}_{\kk\in J},p_\jj(t)) \right]
\\&\quad+ (\rho_\FF\partial_t\vec{u}(t),\vec{u}(t))_{\Omega_\FF}
+ \mathcal L_\FF(\vec{u}(t),p(t);\vec{u}(t),p(t))
\\&
= \mathcal F_\PP(\partial_t\vec{d}(t))
+ \sum_{\jj\in J} \mathcal F_\jj(p_\jj(t))
+ \mathcal F_\FF(\vec{u}(t)).
\end{aligned}\end{equation}
Proceeding as in \cite{corti2022numerical},
we integrate \eqref{eq:proofstab1} w.r.t.~time and we can employ integration by parts in time in the following terms
to obtain
\[\begin{aligned}
\int_0^t (\rho_\PP\partial_{tt}^2\vec{d}(s),\partial_t\vec{d}(s))_{\Omega_\PP} ds
&= \frac{1}{2}\|\sqrt{\rho_\PP}\partial_t\vec{d}(t)\|_{\Omega_\PP}^2
- \frac{1}{2}\|\sqrt{\rho_\PP}\dot{\vec{d}}^0\|_{\Omega_\PP}^2, \\
\int_0^t \mathcal A_\PP(\vec{d}(s),\partial_t\vec{d}(s)) ds
&= \frac{1}{2} \mathcal A_\PP(\vec{d}(t),\vec{d}(t)) - \frac{1}{2} \mathcal A_\PP(\vec{d}^0,\vec{d}^0)\\
\int_0^t (c_\jj\partial_tp_\jj(s),p_\jj(s))_{\Omega_\PP} ds
&= \frac{1}{2}\|\sqrt{c_\jj}p_\jj(t)\|_{\Omega_\PP}^2
- \frac{1}{2}\|\sqrt{c_\jj}p_\jj^0\|_{\Omega_\PP}^2, \qquad\forall\jj\in J, \\
\int_0^t (\rho_\FF\partial_t\vec{u}(s),\vec{u}(s))_{\Omega_\FF} ds
&= \frac{1}{2}\|\sqrt{\rho_\FF}\vec{u}(t)\|_{\Omega_\FF}^2
- \frac{1}{2}\|\sqrt{\rho_\FF}\vec{u}^0\|_{\Omega_\FF}^2.
\end{aligned}\]
Using these identities and the definition of the bulk and broken norms introduced above, we can rewrite the left-hand side of \eqref{eq:proofstab1} (integrated over time) as follows, and use Cauchy-Schwarz's and Young's inequalities on the right-hand side:
\begin{equation}\label{eq:proofstab2}\begin{aligned}
\|&\sqrt{\rho_\PP}\partial_t\vec{d}(t)\|_{\Omega_\PP}^2
 \|\sqrt{\rho_\PP}\dot{\vec{d}}^0\|_{\Omega_\PP}^2
+ \mathcal A_\PP(\vec{d}(t),\vec{d}(t)) - \mathcal A_\PP(\vec{d}^0,\vec{d}^0)
\\&\quad+ \sum_{\jj\in J}\left[\|\sqrt{c_\jj}p_\jj(t)\|_{\Omega_\PP}^2 - \|\sqrt{c_\jj}p_\jj^0\|_{\Omega_\PP}^2
+ 2\int_0^t\left(\mathcal A_\jj(p_\jj(s),p_\jj(s)) +\mathcal C_\jj(\{p_\kk(s)\}_{\kk\in J},p_\jj(s)) \right)
ds\right]
\\&\quad+ \|\sqrt{\rho_\FF}\vec{u}(t)\|_{\Omega_\FF}^2 - \|\sqrt{\rho_\FF}\vec{u}^0\|_{\Omega_\FF}^2
+ 2\int_0^t\mathcal L_\FF(\vec{u}(s),p(s);\vec{u}(s),p(s))ds
\\& \leq \int_0^t\left(\left\|\frac{1}{\sqrt{\rho_\PP}}\vec{f}_\PP(s)\right\|_{\Omega_\PP}^2+\|\sqrt{\rho_\PP}\partial_t\vec{d}(s)\|_{\Omega_\PP}^2 +\sum_{\jj\in J}\left\|\frac{1}{\sqrt{c_\jj}}g_\jj(s)\right\|_{\Omega_\PP}^2
\right.\\&\left .\qquad\qquad +\sum_{\jj\in J}\|\sqrt{c_\jj}p_\jj(s)\|_{\Omega_\PP}^2
+ \left\|\frac{1}{\sqrt{\rho_\FF}}\vec{f}_\FF(s)\right\|_{\Omega_\FF}^2+\|\sqrt{\rho_\FF}\vec{u}(s)\|_{\Omega_\FF}^2\right)ds.
\end{aligned}\end{equation}

We now consider continuity and coercivity results for the bilinear forms appearing in \eqref{eq:proofstab2} with respect to the broken norms \eqref{eq:brokennorms}.
These results, proven in \cite{corti2022numerical,AMVZ22}, are reported in \cref{sec:contcoerc} (\cref{th:contcoerc}) and they include an inf-sup condition for the form $\mathcal B_\FF$, with a constant $\beta_{\FF,h}$ independent of the mesh elements size.
According to \cite
{AMVZ22}, if $\gamma_\vec{v}$ is large enough, there exists $\alpha>0$ such that $\alpha=\mathcal O(\beta_{\FF,h}^{-2})$ and
\[
\mathcal L_\FF(\vec{u}(s),p(s);\vec{u}(s),p(s)) \geq \alpha (\normDGu{\vec{u}(s)}^2 + \normDGp{p(s)}^2).
\]
Analogously, the coercivity results of the forms $\mathcal{A}_\PP, \mathcal A_\jj, \mathcal C_\jj$ and the continuity of $\mathcal A_\PP$ yield \cite
{corti2022numerical}
\[\begin{aligned}
\mathcal A_\PP(\vec{d}(t),&\vec{d}(t)) - \mathcal A_\PP(\vec{d}^0,\vec{d}^0) + 2\sum_{\jj\in J}\int_0^t\left(\mathcal A_\jj(p_\jj(s),p_\jj(s)) +\mathcal C_\jj(\{p_\kk(s)\}_{\kk\in J},p_\jj(s)) \right) ds
\\& \gtrsim \normDGd{\vec{d}(t)}^2 -\normDGd{\vec{d}^0}^2 + \sum_{\jj\in J}\int_0^t\left(\normDGj{p_\jj(s)}^2+\|\sqrt{\beta_\jj^\text{e}}p_\jj(s)\|_{\Omega_\PP}^2\right) ds.
\end{aligned}\]
We can use these results on the left-hand side of \eqref{eq:proofstab2} to obtain
\[
\begin{aligned}
\|\sqrt{\rho_\PP}\partial_t\vec{d}(t)\|_{\Omega_\PP}^2
&+ \normDGd{\vec{d}(t)}^2 + \sum_{\jj\in J}\left[\|\sqrt{c_\jj}p_\jj(t)\|_{\Omega_\PP}^2 + \int_0^t\left(\normDGj{p_\jj(s)}^2+\|\sqrt{\beta^\text{e}_\jj}p_\jj(s)\|_{\Omega_\PP}^2\right)ds\right]
\\& \qquad -\|\sqrt{\rho_\PP}\dot{\vec{d}}^0\|_{\Omega_\PP}^2 - \normDGd{\vec{d}^0}^2 - \sum_{\jj\in J}\|\sqrt{c_\jj}p_\jj^0\|_{\Omega_\PP}^2
\\& \qquad
+ \normENfluid{(\vec{u}(t),p(t))}^2
- \|\sqrt{\rho_\FF}\vec{u}^0\|_{\Omega_\FF}^2
\\& \lesssim \int_0^t\left(\left\|\frac{1}{\sqrt{\rho_\PP}}\vec{f}_\PP(s)\right\|_{\Omega_\PP}^2
+ \sum_{\jj\in J}\left\|\frac{1}{\sqrt{c_\jj}}g_\jj(s)\right\|_{\Omega_\PP}^2
+ \left\|\frac{1}{\sqrt{\rho_\FF}}\vec{f}_\FF(s)\right\|_{\Omega_\FF}^2\right) ds.
\end{aligned}\]
The definitions of $\normEN{\cdot}$ and $\normENzero{\cdot}$ and the fact that $\rho_\PP,c_\jj,\rho_\FF$ are constant conclude the proof.
\end{proof}

We now proceed to derive an a-priori error estimate for the error introduced by the space discretization of the problem.
For any $t\in(0,T]$, let $(\vec{d},\{p_\jj\}_{\jj\in J},\vec{u},p)$ 
denote the weak solution of problem \eqref{eq:weak} and let
$(\vec{d}_h,\{p_{\jj,h}\}_{\jj\in J},\vec{u}_h,p_h)$
denote the semidiscrete solution of problem \eqref{eq:DG}, obtained with sufficiently large stability parameters as defined in \eqref{eq:penaltyparams}.

Using the inverse trace inequality \cite{antonietti2013hp,antonietti2016review}
\begin{equation}\label{eq:inversetrace}
    \|\phi\|_{\partial K}^2 \lesssim
    h_K^{-1}
    \|\phi\|_K^2 \qquad \forall \phi\in \mathbb P^m(K), K\in\mathscr T_h,
\end{equation}
we can state the following continuity result for the interface forms, whose proof is provided in \cref{sec:proofnewcontJ}:
\begin{lemma}\label{th:newcontJ}
    Under \cref{hp:mesh},
    the following inequalities hold:
    \[
    \begin{aligned}
    |\mathcal J(q,\vec{w}_h,\vec{v}_h)| & \lesssim \|\eta^{1/2}q\|_{\mathscr F_h^\Sigma}\normDGd{\vec{w}_h} + \|\gamma_p^{-1/2}q\|_{\mathscr F_h^\Sigma}\normDGu{\vec{v}_h}
    \\&\qquad\qquad
    \forall q\in H^2(\mathscr T_{h,\PP}), \vec{w}_h\in\spaceW, \vec{v}_h\in\spaceV,\\[1ex]
    |\mathcal J(q_h,\vec{w},\vec{v})| & \lesssim \normDGp{q_h}\|\eta^{1/2}\vec{w}\|_{\mathscr F_h^\Sigma} + \normDGp{q_h}\|\gamma_\vec{v}^{1/2}\vec{v}\|_{\mathscr F_h^\Sigma}
    \\&\qquad\qquad
    \forall  q_h\in\spaceQE, \vec{w}\in [H^2(\mathscr T_{h,\FF})]^d, \vec{v}\in [H^2(\mathscr T_{h,\FF})]^d.
    \end{aligned}
    \]
\end{lemma}

We introduce the following additional norms for non-discrete functions:
\[\begin{aligned}&\begin{aligned}
\normcontD{\vec{w}}^2 &= \normDGd{\vec{w}}^2 + \|\eta^{-1/2}\average{\sigma_\PP(\vec{w})}\|_{\mathscr F_{\PP,h}^\text{I}\cup\mathscr F_{\PP,h}^\text{D}}^2
&\qquad\forall \vec{w}\in [H^2(\mathscr T_{h,\PP})]^d,\\
\normcontJ{q_\jj}^2 &= \normDGj{q_\jj}^2 + \|\zeta^{-1/2}\average{\frac{1}{\mu_\jj}K_\jj\nabla_hq_\jj}\|_{\mathscr F_{\PP,h}^\text{I}\cup\mathscr F_{\PP,h}^{\text{D}_\jj}}^2
&\qquad\forall q_\jj\in H^2(\mathscr T_{h,\PP}),
\qquad \forall\jj\in J,\\
\normcontU{\vec{v}}^2 &= \normDGu{\vec{v}}^2 + \|\gamma_{\vec{v}}^{-1/2}\average{\sigma_\FF(\vec{v})}\|_{\mathscr F_{\FF,h}^\text{I}\cup\mathscr F_{\FF,h}^\text{D}}^2
&\qquad\forall \vec{v}\in [H^2(\mathscr T_{h,\FF})]^d,\\
\normcontP{q}^2 &= \normDGp{q}^2 + \|{\gamma_p^{1/2}}\average{q}\|_{\mathscr F_{\FF,h}^\text{I}}^2
&\qquad\forall q\in H^1(\mathscr T_{h,\FF}),
\end{aligned}\\
&\normcont{(\vec{w},\{q_\jj\}_{\jj\in J},\vec{v},p)}^2 = \normcontD{\vec{w}}^2+\sum_{\jj\in J}\normcontJ{q_\jj}^2 + \normcontU{\vec{v}}^2+\normcontP{q}^2.
\end{aligned}\]

Denoting by $\mathscr E_K: H^s(\Omega)\to H^s(\mathbb R^d)$ the Stein extension operator for a Lipschitz domain $\Omega$ defined in \cite{stein1970singular},
the following interpolation result can be stated
(cf.~\cite{corti2022numerical,AMVZ22,antonietti2023discontinuous,perTraceInterp}):
\begin{lemma}\label{th:interp}
Under
\cref{hp:mesh},
the following estimates hold:
\[\begin{aligned}
    &\forall (\vec{w},\{q_\jj\}_{\jj\in J},\vec{v},q)\in [H^{m+1}(\mathscr T_{h,\PP})]^{d+N_J}\times[H^{m+1}(\mathscr T_{h,\FF})]^{d+1}\\
    &\exists (\vec{w}_I,\{q_{\jj I}\}_{\jj\in J},\vec{v}_I,q_I)\in\spaceW\times\bigtimes_{\jj\in J}\spaceQj\times\spaceV\times\spaceQ \quad \text{such that}
    \\
    &i)\ \normcont{(\vec{w}-\vec{w}_I,\{q_\jj-q_{\jj I}\}_{\jj\in J},\vec{v}-\vec{v}_I,q-q_I)}^2 \\
    &\qquad\lesssim \sum_{K\in\mathscr T_{h,\PP}}h_K^{2m}\left(\|\mathcal E_K\vec{w}\|_{[H^{m+1}(\widehat{K})]^d}^2+\sum_{\jj\in J}\|\mathcal E_\jj q_\jj\|_{H^{m+1}(\widehat{K})}^2+\|\mathcal E_K\vec{d}\|_{[H^{m+1}(\widehat{K})]^d}^2+\|\mathcal E_Kp\|_{H^{m+1}(\widehat{K})}^2\right), \\
    &ii)\ \|\vec{w}\|_{\mathscr F^\Sigma}^2+\sum_{\jj\in J}\|q_\jj\|_{\mathscr F^\Sigma}^2 + \|\vec{v}\|_{\mathscr F^\Sigma}^2+\|q\|_{\mathscr F^\Sigma}^2\\
    &\qquad\lesssim \sum_{K\in\mathscr T_{h,\PP}}h_K^{2m+1}\left(\|\mathcal E_K\vec{w}\|_{[H^{m+1}(\widehat{K})]^d}^2+\sum_{\jj\in J}\|\mathcal E_\jj q_\jj\|_{H^{m+1}(\widehat{K})}^2+\|\mathcal E_K\vec{d}\|_{[H^{m+1}(\widehat{K})]^d}^2+\|\mathcal E_Kp\|_{H^{m+1}(\widehat{K})}^2\right),
\end{aligned}\]
where $\widehat{K}\supseteq K$, for each $K\in\mathscr T_h$, are shape-regular simplexes as in \cref{hp:mesh}.
\end{lemma}

Combining the results above, we can prove the following optimal convergence result, whose proof is provided in \cref{sec:proofConv}.
\begin{theorem}[A priori error estimate]\label{th:conv}
    Under the same assumptions of \cref{th:stab}, if the solution of problem \eqref{eq:weak}
is sufficiently regular,
    and the initial conditions $\vec{d}^0, \dot{\vec{d}}^0, \vec{u}^0, p_\jj^0,\jj\in J$
are sufficiently regular,
    the following estimate holds for each $t\in (0,T]$ and for each $m\geq 1$:
    \begin{equation}\label{eq:error}
    \begin{aligned}
    &
    \normEN{(\vec{e}^\vec{d},\{e^{P_\jj}\}_{\jj\in J},\vec{e}^\vec{u},e^{P_\FF})}^2
    \\ &\qquad
    \lesssim \sum_{K\in\mathscr T_{h,\PP}}
    {h_K^{2m}}\left\{
    \|\mathcal E_K\vec{d}(t)\|_{[H^{m+1}(\widehat{K})]^d}^2 + \sum_{\kk\in J}\|\mathcal E_K p_\kk(t)\|_{H^{m+1}(\widehat{K})}^2
    \right.
    \\ &\qquad
    \qquad\qquad\left.
    +\int_0^t\left[
    \|\mathcal E_K\partial_t\vec{d}(s)\|_{[H^{m+1}(\widehat{K})]^d}^2+\|\mathcal E_K\partial_{tt}^2\vec{d}(s)\|_{[H^{m+1}(\widehat{K})]^d}^2
    \right]ds
    \right.
    \\ &\qquad
    \qquad\qquad\left.
    +\int_0^t\left.
    \sum_{\kk\in J}\left(
    \|\mathcal E_Kp_\kk(s)\|_{H^{m+1}(\widehat{K})}^2+\|\mathcal E_K\partial_tp_\kk(s)\|_{H^{m+1}(\widehat{K})}^2
    \right)
    \right.ds
    \right\}
    \\ &\qquad
    \quad+ \sum_{K\in\mathscr T_{h,\FF}}
    {h_K^{2m}}\left.
    \int_0^t\left[
    \|\mathcal E_K\vec{u}(s)\|_{[H^{m+1}(\widehat{K})]^d}^2 + \|\mathcal E_K\partial_t\vec{u}(s)\|_{[H^{m+1}(\widehat{K})]^d}^2
    \right.\right.\\ &\qquad\qquad\qquad\qquad\qquad\left.\left.\phantom{\|^2_{H^{m+1}}}
    +\|\mathcal E_Kp(s)\|_{H^{m+1}(\widehat{K})}^2
    \right]ds,
    \right.
    \end{aligned}
    \end{equation}
    where $\vec{e}^\vec{d}=\vec{d}-\vec{d}_h, e^{P_\jj}=p_\jj-p_{\jj,h}\ \forall \jj\in J, \vec{e}^\vec{u}=\vec{u}-\vec{u}_h, e^{P_\FF}=p - p_h$,
and $\widehat{K}\supseteq K$, for each $K\in\mathscr T_h$, are shape-regular simplexes as in \cref{hp:mesh}.
\end{theorem}

\section{Fully discrete problem}\label{sec:fullydiscrete}
We introduce a uniform partition $\{t_n\}_{n=0}^N$ of the interval $(0,T]$, with constant timestep $\Delta t=t_{n+1}-t_n$, for all $n=0,\ldots N-1$.
Starting from the algebraic form $\eqref{eq:formalalg}$ of problem $\eqref{eq:DG}$, we discretize the elastic momentum equation (first row) with Newmark's $\beta$-method, whereas we employ the $\theta$-method to discretize all the compartment pressure equations and the fluid problem.
For Newmark's discretization, we introduce two auxiliary vector variables $\vec{Z}^n, \vec{A}^n$ representing the expansion coefficients of the approximate elastic velocity and acceleration at time $t_n$ \cite{corti2022numerical}.
The resulting algebraic problem has the form
\begin{equation}\label{eq:fullydiscr}
A_1 \vec{X}^{n+1} = A_2\vec{X}^n + \vec{F}^{n+1}, \qquad n=1,\dots,N,
\end{equation}
where
\begin{align}
\vec{X}^n &= \begin{bmatrix}
\vec{D}^n \\ \vec{Z}^n \\ \vec{A}^n \\ \vec{P}_\text{A}^n \\ \vdots \\ \vec{P}_\EE^n \\ \vec{U}^n \\ \vec{P}^n
\end{bmatrix},
\qquad\qquad\qquad
\vec{F}^n = \begin{bmatrix}
\vec{F}_\PP^n \\
\vec{0} \\
\vec{0} \\
\theta\vec{F}_\text{A}^{n+1}+(1-\theta)\vec{F}_\text{A}^n \\ \vdots \\ \theta\vec{F}_\EE^{n+1}+(1-\theta)\vec{F}_\EE^n \\ \theta\vec{F}_\FF^{n+1}+(1-\theta)\vec{F}_\FF^n \\ \vec{0}
\end{bmatrix},
\end{align}
and the expression of the matrices $A_1,A_2$ are reported in \cref{sec:matfullydiscrete}.

\section{Verification tests -- convergence}\label{sec:results}

In this section, we verify the theoretical error bounds of \cref{th:conv}.
In all tests, we consider only one pressure compartment, that is $\jj\in J=\{E\}$.
The constant coefficients of the penalty parameters defined in \eqref{eq:penaltyparams} are set as $\overline{\eta}=\overline{\zeta}_\EE=\overline{\gamma}_\vec{v}=\overline{\gamma}_p=10$.
Regarding time discretization, all the results were obtained with $\beta=0.25$ and $\gamma=0.5$ for the Newmark scheme and $\theta=0.5$ for the $\theta$-method, both for Stokes' problem and the pressure compartment of the poro-elastic system.

The tests have been implemented in the 2D version of \lymph{} \cite{lymph}, an open-source MATLAB library for the solution of multiphysics problems with the PolyDG method, developed at MOX, Department of Mathematics, Politecnico di Milano.

\subsection{Test case 1: steady solution}\label{sec:convStaz}

In this test, we consider a manufactured steady solution of problem \eqref{eq:NSMPE}-\eqref{eq:bd}-\eqref{eq:interf}, so that we can verify the convergence of the semidiscrete formulation without spoiling the results with time discretization error.
In particular, we introduce the following exact solution on the 2D domain $\Omega=\Omega_\PP\cup\Omega_\FF=(-1,0)\times(0,1)\cup(0,1)\times(0,1)$, with interface $\Sigma=\{0\}\times(0,1)$ depicted in \cref{fig:domain}:
\begin{equation}\label{eq:exact2Dsteady1cpt}\begin{aligned}
    \vec{d}^\text{steady} &= \pi\mu\frac{K_\EE}{\mu_\EE^2}(1-\alpha_\EE)(\cc-\ss)\begin{bmatrix} -1 \\ 1 \end{bmatrix}, \\
    p_\EE^\text{steady} &= -\pi x \cos(\pi y) - 2\pi^2\mu\frac{K_\EE}{\mu_\EE}\sin(\pi y),\\
    \vec{u}^\text{steady} &= \pi\frac{K_\EE}{\mu_\EE}(\cc-\ss)\begin{bmatrix} 1 \\ -1\end{bmatrix}, \\
    p^\text{steady} &= -x\cos(\pi y) - 4\pi^2\mu\frac{K_\EE}{\mu_\EE}\sin(\pi y).
\end{aligned}\end{equation}
The source terms of $\eqref{eq:NSMPE}$ corresponding to the exact solution above are
\[\begin{aligned}
    \vec{f}_\PP &= \begin{bmatrix}
        -2\pi^3\mu\frac{K_\EE}{\mu_\EE}(1-\alpha_\EE)(\cc-\ss) -\pi\alpha_\EE\cos(\pi y) \\
        2\pi^3\mu\frac{K_\EE}{\mu_\EE}[(1-\alpha_\EE)(\cc-\ss) - \alpha_\EE\cos(\pi y)] + \pi^2\alpha_\EE x\sin(\pi y)
    \end{bmatrix},\\
    g_\EE &= \left(\pi^2\frac{K_\EE}{\mu_\EE} + \beta_\EE^\text{e}\right)
        p_\EE,\\
    \vec{f}_\FF &= \begin{bmatrix}
        2\pi^3\mu\frac{K_\EE}{\mu_\EE}(\cc-\ss) - \cos(\pi y) \\
        -2\pi^3\mu\frac{K_\EE}{\mu_\EE}(\cc-\ss + 2\cos(\pi y)) + pi x\sin(\pi y)
    \end{bmatrix},
\end{aligned}\]
and the Neumann stress on $\Gamma_\text{out}$ is:
\[
-\overline{p}^\text{out}\vec{n}_\FF=\left(\cos(\pi y) + 6\pi^2\mu\frac{K_\EE}{\mu_\EE}\sin(\pi y)\right)\vec{n}_\FF.
\]

With these data, we performed spatial convergence tests setting $\alpha_\EE=0.5$ and all the physical parameters $\rho_\PP,\mu_\PP,\lambda_\PP,c_\EE,\mu_\EE,K_\EE,\beta_\EE^\text{e},\rho_\FF,\mu_\FF$ equal to 1.
In \cref{fig:conv2D} we report the computed error.
From the results, we can clearly observe that the error decays w.r.t.~$h$ with a rate of order $m$, as predicted by our theoretical estimate \eqref{eq:error}.
Moreover, we can observe spectral convergence of the error w.r.t.~the polynomial approximation degree $m$.

\begin{figure}
    \centering
    \includegraphics[width=0.475\textwidth]{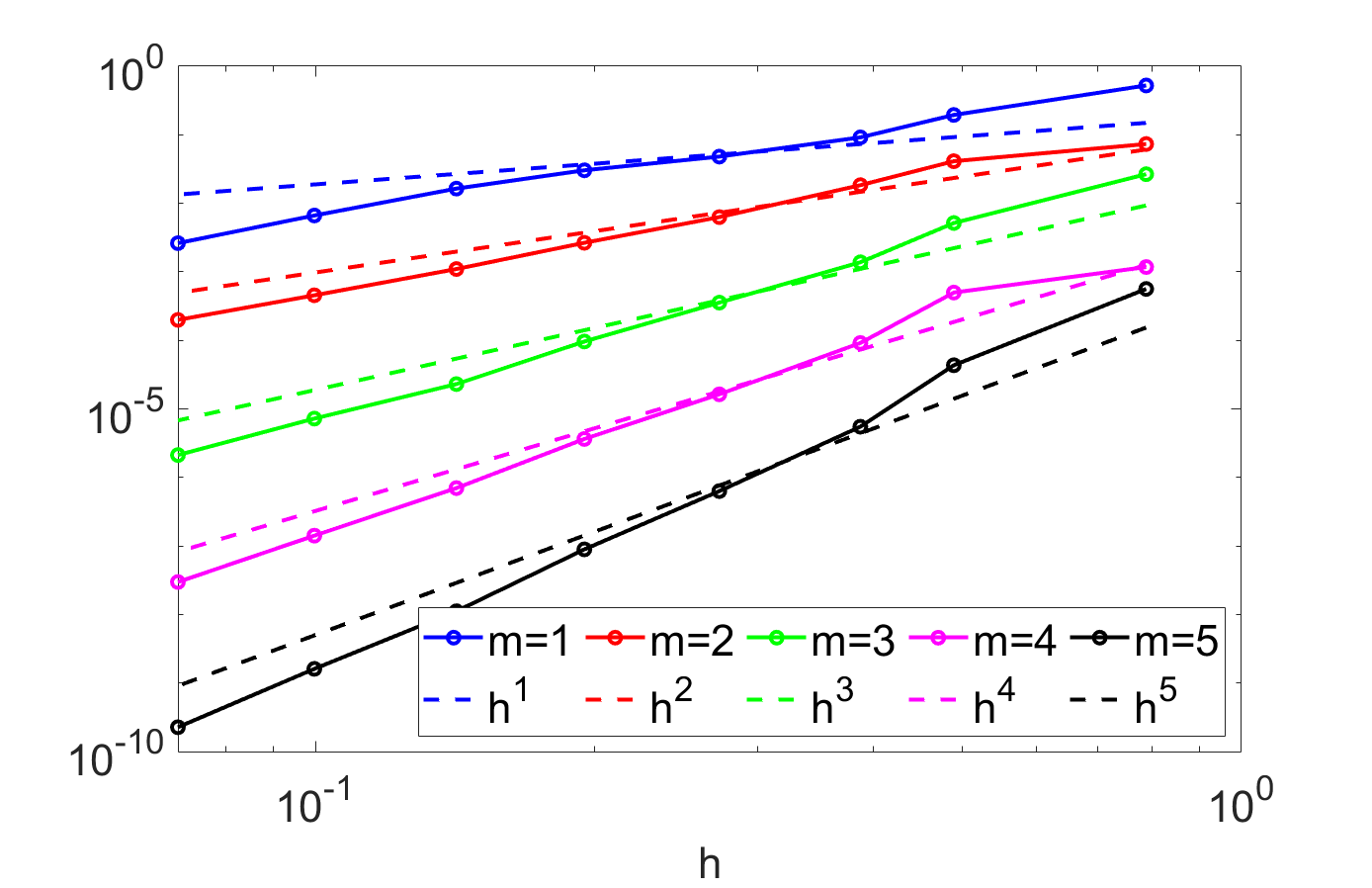}
    \includegraphics[width=0.475\textwidth]{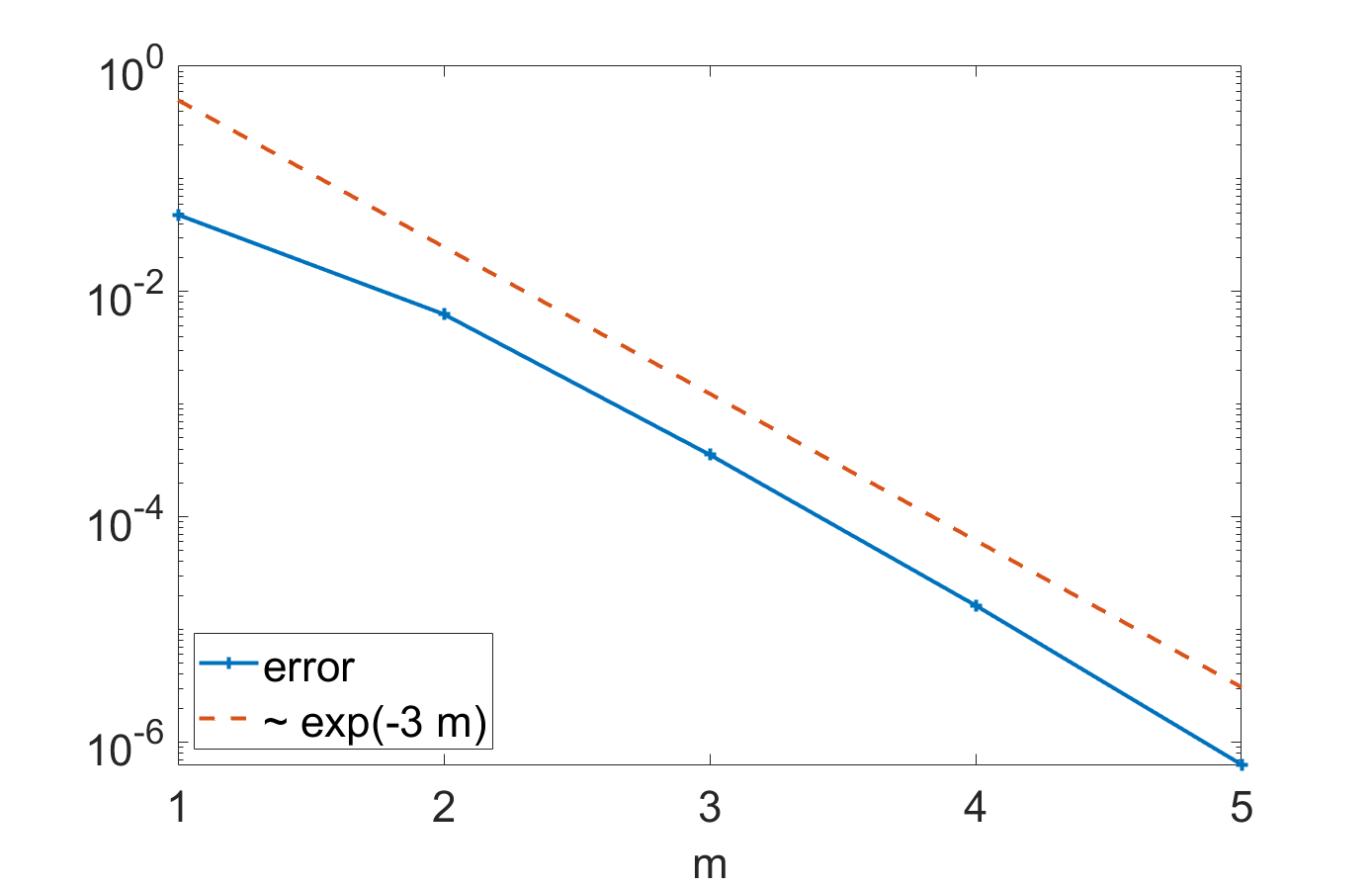}
    \caption{Test case 1. Left: computed errors in the energy norm \eqref{eq:normcoerc} versus $h$, for different choices of the polynomial degree $m=1,2,3,4,5$. Right: computed errors in the energy norm \eqref{eq:normcoerc} versus $m$, with $h=0.273$ (corresponding to $N=80$ polygons).}
    \label{fig:conv2D}
\end{figure}

\subsection{Test case 2: unsteady solution}\label{sec:convTimedep}
We consider the rectangular domain $\Omega=\Omega_\PP\cup\Omega_\FF=(-1,0)\times(0,1)\cup(0,1)\times(0,1)$ of \cref{sec:convStaz} and we introduce the following functions of time:
\begin{align*}
g_\PP(t) &= \cos(\eta t)-\sin(\eta t), & \text{where }\eta = \frac{\mu_\PP}{\mu_\FF(1-\alpha_\EE)},\\
g_\vec{u}(t) &= g_\PP(t) - \frac{\dot{g}_\PP(t)}{\eta}, & g_p(t) = \frac{g_\EE(t)+g_\vec{u}(t)}{2}.
\end{align*}
A time-dependent exact solution of problem \eqref{eq:NSMPE}-\eqref{eq:bd}-\eqref{eq:interf} can be manufactured combining these functions with the steady solution introduced in \cref{sec:convStaz}:
\begin{equation}\label{eq:exact2Dunsteady}\begin{aligned}
    \vec{d}(x,y,t) &= g_\PP(t)\vec{d}^\text{steady}(x,y),
    &
    p_\EE(x,y,t) &= g_\PP(t)p_\EE^\text{steady}(x,y), \\
    \vec{u}(x,y,t) &= g_\vec{u}(t)\vec{u}^\text{steady}(x,y),
    &
    p(x,y,t) &= g_p(t)p^\text{steady}(x,y),
\end{aligned}\end{equation}
with suitable definitions of the source terms and boundary data.
Again, we performed convergence tests setting $\alpha_\EE=0.5$ andall other physical parameters to 1.
We choose a time step $\Delta t=\SI{1e-3}{\second}$ and a final time  $T=5\,\Delta t$.
We report the computed errors in \cref{fig:conv2Dunsteady} (log-log scale).
As predicted by \cref{th:conv}, we observe that the error in the energy norm decays at a rate proportional to $h^m$, for any $m\geq1$.
Moreover, even though not covered by our theoretical analysis, we observe exponential convergence of the error for fixed $h$ and increasing $m$.
Finally, we observe that the value chosen for $\Delta t$ is small enough to prevent error saturation for the sequence of meshes considered in this test.

\begin{figure}
    \centering
    \includegraphics[width=0.475\textwidth]{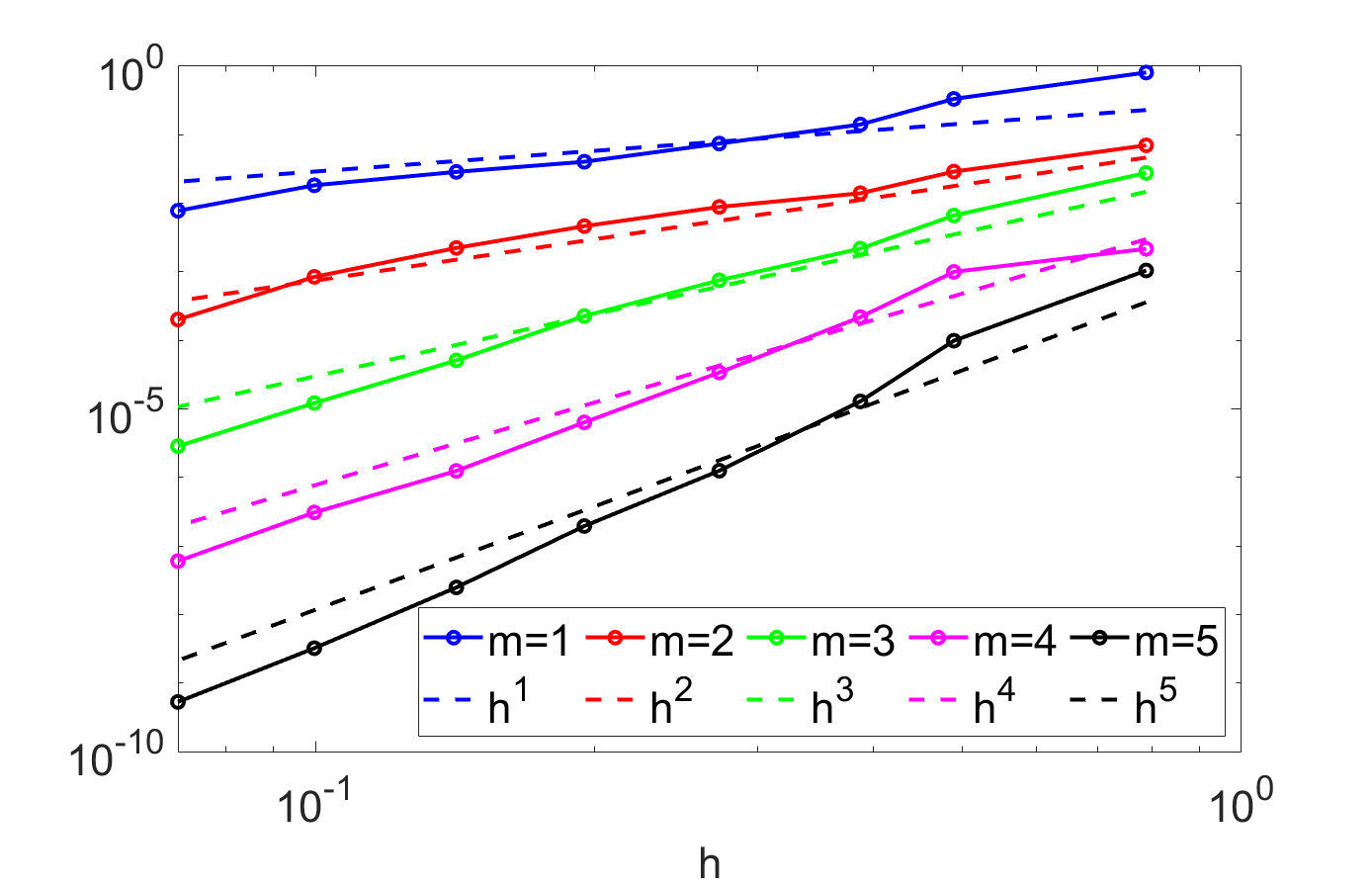}
    \includegraphics[width=0.475\textwidth]{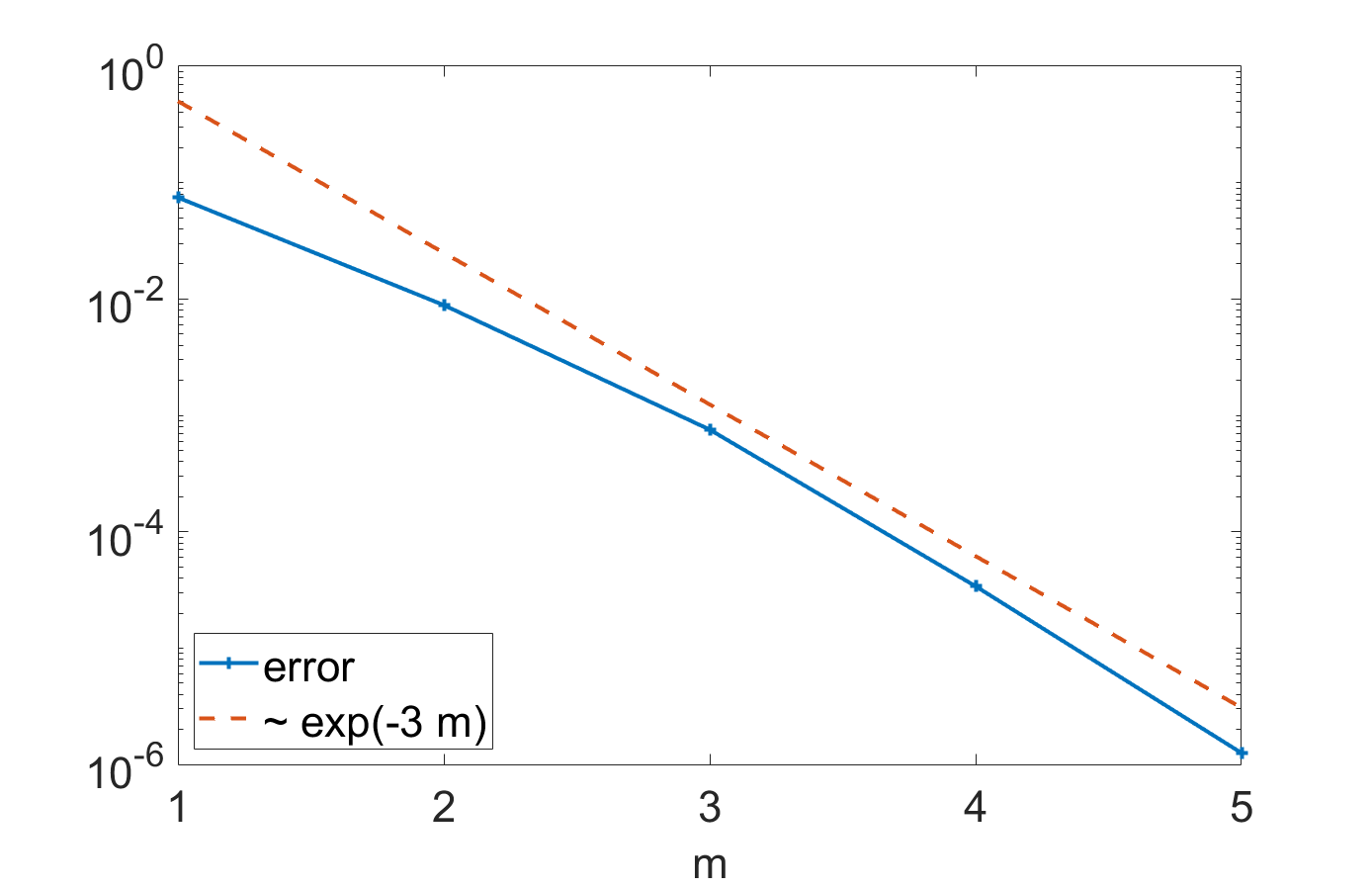}
    \caption{Test case 2. Left: computed errors in the energy norm \eqref{eq:normcoerc} versus $h$, with different polynomial degrees $m$ (same for all variables). Right: computed errors in the energy norm \eqref{eq:normcoerc} versus $m$, with $h=0.273$ corresponding to $N=80$ polygons.}
    \label{fig:conv2Dunsteady}
\end{figure}

\section{Numerical results on a 2D slice of the brain}\label{sec:brain}
\begin{figure}
    \centering
    $\vcenter{\hbox{\includegraphics[width=0.4\textwidth]{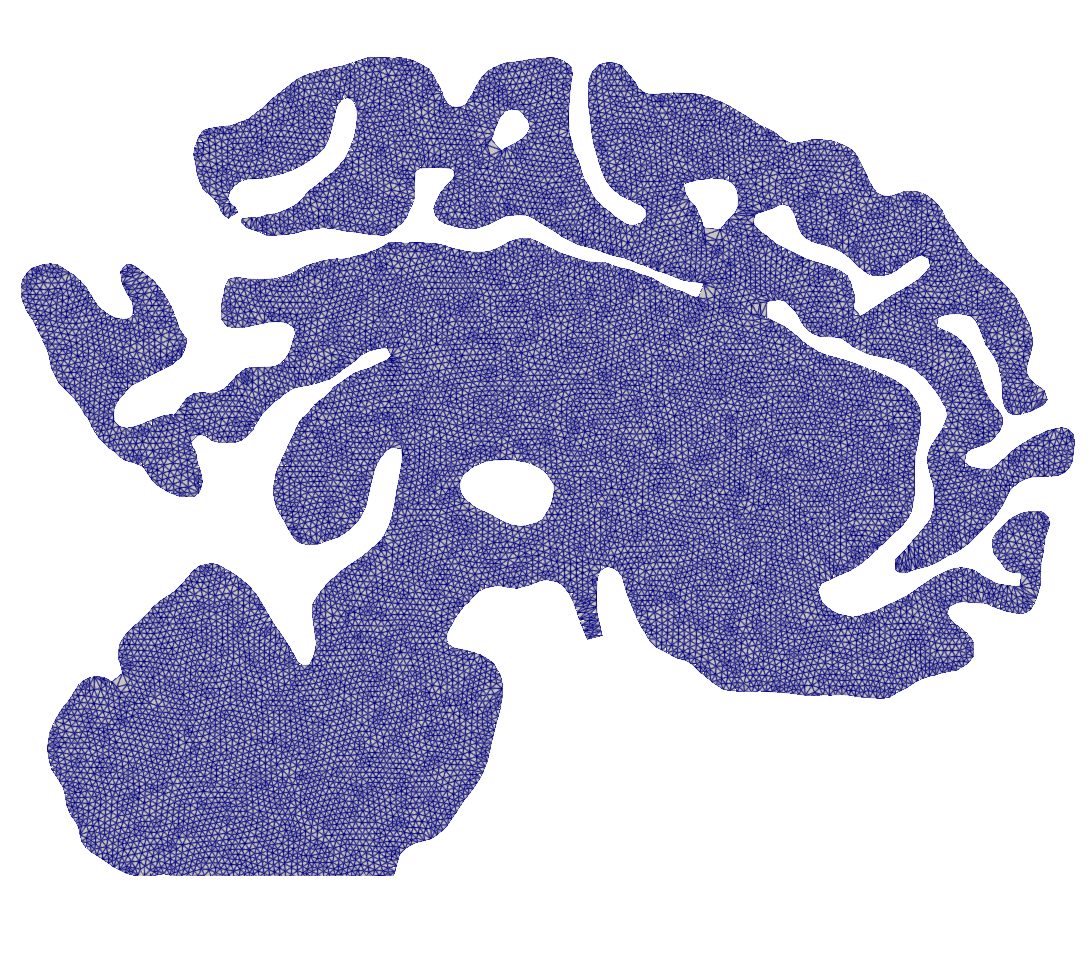}}}$
    \qquad\qquad
    $\vcenter{\hbox{\setlength{\fboxsep}{0pt}\fbox{\includegraphics[width=0.27\textwidth]{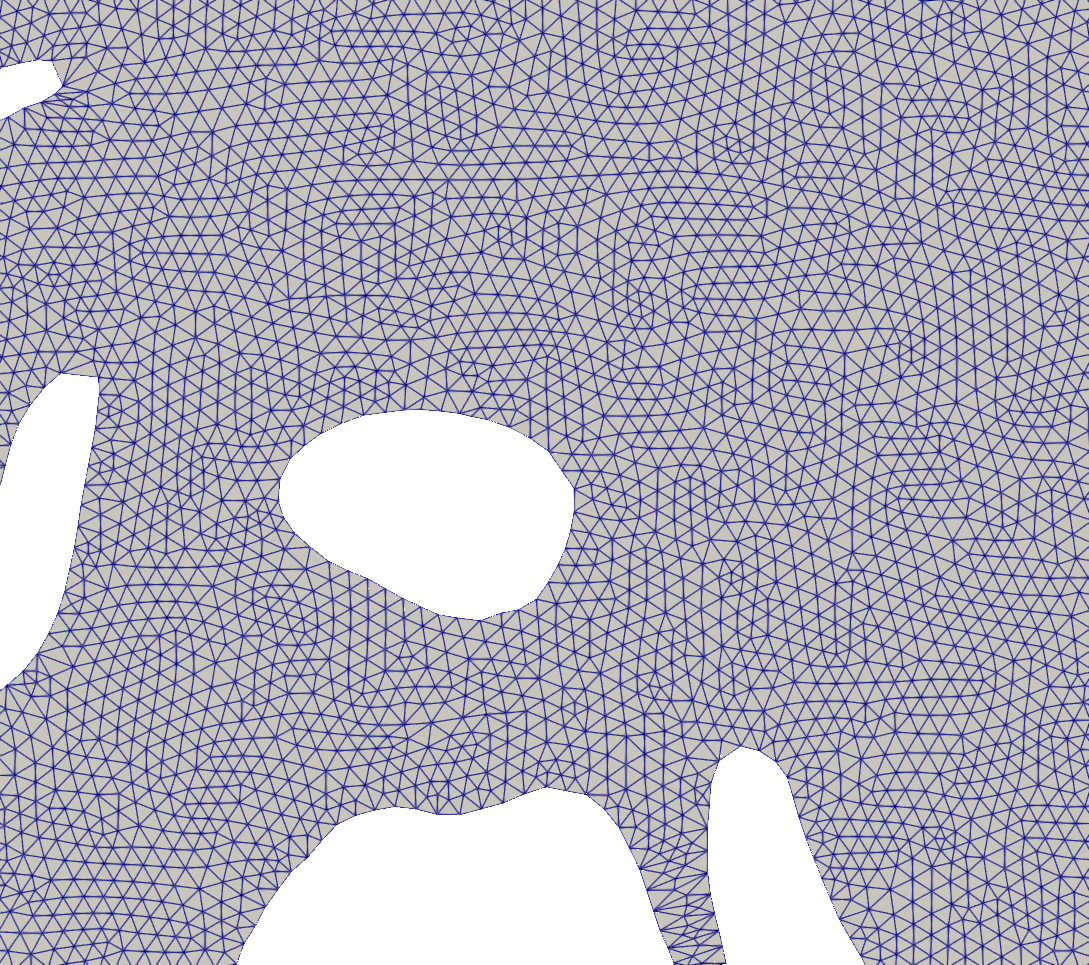}}}}$
    \\
    \includegraphics[width=0.435\textwidth]{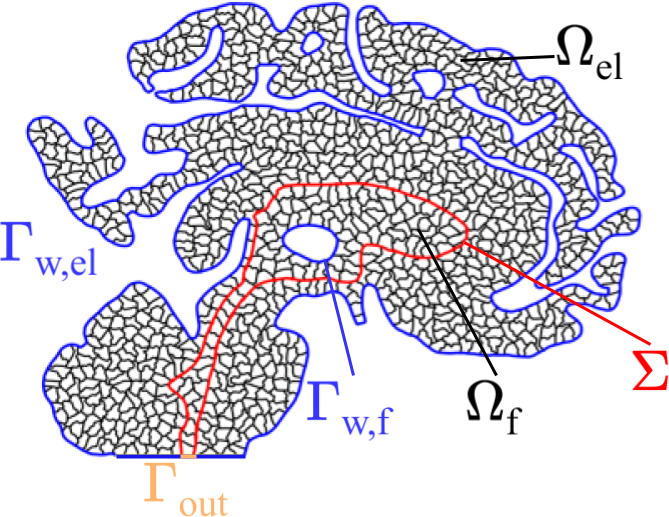}
    \qquad\qquad
    \includegraphics[width=0.33\textwidth]{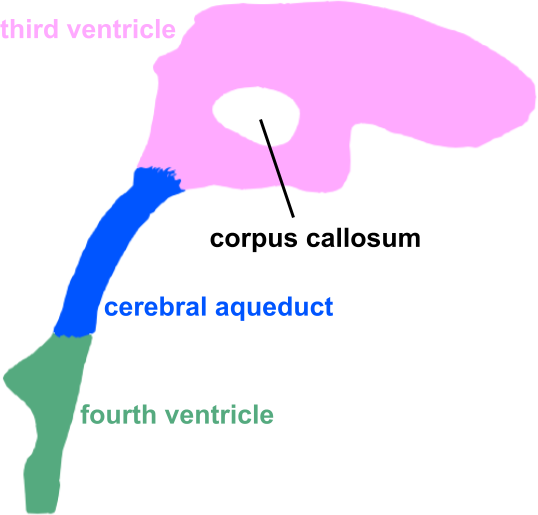}
    \caption{
    Triangular mesh of the 2D slice reconstructed from MRI (top left; zoom on top right to see the triangular elements),  polygonal computational mesh obtained by agglomeration (bottom left), and topographic anatomy of the fluid domain $\Omega_\FF$ (bottom right).}
    \label{fig:brainMesh}
\end{figure}
In this section, we demonstrate the capability of the proposed method for the solution of a realistic problem on a 2D slice of the brain and its ventricles, shown in \cref{fig:brainMesh}.
The geometry of the problem is based on structural Magnetic Resonance Images (MRI)
available in the OASIS-3 database (\url{https://oasis-brains.org}) \cite{lamontagne2019oasis}.
By means of Freesurfer (\url{https://surfer.nmr.mgh.harvard.edu/}) \cite{fischl2012freesurfer}, a three-dimensional brain geometry was segmented, and then sliced along the sagittal plane by VMTK (\url{vmtk.org}) \cite{antiga2008image}.
The resulting triangular 2D meshes of the cerebral tissue and of the brain ventricles surrounded by it are composed of 25847 and 3286 elements, respectively (see \cref{fig:brainMesh}).
Yet, the flexibility of the PolyDG method allows to employ meshes with elements of generic shape: by agglomeration, we can thus considerably reduce the number of mesh elements while retaining the same geometrical detail of the original triangular grid.
We agglomerate the grid by means of ParMETIS (\url{https://github.com/KarypisLab/ParMETIS}) \cite{karypis1997parmetis}, obtaining the polygonal mesh shown in \cref{fig:brainMesh}, consisting of 910 elements in $\mathscr T_\PP$ and 101 in $\mathscr T_\FF$.
This agglomeration is performed separately for the two physical domains, to preserve geometrical accuracy at the interface $\Sigma$.

Aiming at reproducing conditions in the physiological regime, we set the data and parameters of problem \eqref{eq:NSMPE} as follows.
We split the Dirichlet boundary $\Gamma_\text{w}$ into $\Gamma_{\text{w},\PP}$ representing the dura mater membrane surrounding the brain tissue and the boundary $\Gamma_{\text{w},\FF}$ of the corpus callosum (the whole in the center of \cref{fig:brainMesh}).
In the poroelastic problem, we consider only the compartment related to the extracellular CSF, namely $J=\{\EE\}$, and we assume no flow ($\nabla p_\EE\cdot\vec{n}_\PP=0$) through the dura mater $\Gamma_{\text{w},\PP}$.
The source of extracellular CSF is given by
$
    g_\EE = \SI{2e-3}{}\,\pi\sin(2\pi t) \SI{}{\per\second},
$
representing the variations in the interstitial CSF due to blood pulsation \cite{lee2019mixed,corti2022numerical,causemann2022human}, with a period of $\SI{1}{\second}$ representing a heartbeat.
No additional external forces act on the poroelastic tissue or the ventricle flow, that is $\vec{f}_\PP=\vec{f}_\FF=\vec{0}$.
No slip conditions $\vec{u}=\vec{0}$ are enforced on $\Gamma_{\text{w},\FF}$, while a no-stress condition $\overline{p}^\text{out}=0$ is imposed on the obex $\Gamma_\text{out}$ of the fourth ventricle, where the CSF would flow into the central canal of the spinal cord.
Zero initial conditions are imposed on all variables and the values of the remaining physical parameters of the problem are reported in \cref{tab:modelparams}, with the slight modifications $\alpha_\EE=0.49, \beta_\EE^\text{e}=\SI{0}{\square\meter\per\newton\per\second}$: these values are in the physiological range of the brain function \cite{lee2019mixed,corti2022numerical,causemann2022human}.
For the PolyDG method we employ a polynomial degree $m=2$ for all variables, while for the time advancement based on Newmark's ($\beta=0.25, \gamma=0.5$) and Crank-Nicolson's ($\theta=0.5$) methods we use a time step $\Delta t=\SI{1e-2}{\second}$ from $t_0=\SI{0}{\second}$ to $T=\SI{1}{\second}$.

Selected snapshots of the computed solutions are reported in \cref{fig:brainDP,fig:brainUP}.
Due to the choice of a zero outlet pressure $\overline{p}^\text{out}$, pressures $p_\EE, p$ should be interpreted as \emph{pressure differences} (w.r.t.~the fourth ventricle outflow) rather than absolute pressure values.
We notice that the amplitude of the brain displacement $\vec{d}$ is always below $\SI{0.2}{\milli\meter}$, thus justifying the choice of a linear elasticity modeling of the tissue, at least in the current settings.
Due to such small displacements, the interstitial pressure $p_\EE$ at the interface and the ventricle CSF pressure $p$ are substantially equal through the whole simulation timespan, in accordance with the interface conditions \eqref{eq:interf}.
The ventricle pressure $p$ undergoes very small variations, which are better observable in \cref{fig:brainUP}.
We can see that the CSF circulates in the third ventricle around the corpus callosum, whereas its flow through the cerebral aqueduct is always directed outwards, thus fulfilling its clearance function.
This clearance occurs notwithstanding the fact that the source term $g_\EE$ is negative in the second half of the timespan, and consequently the interstitial pressure $p_\EE$ is lower than the outlet pressure $\overline{p}^\text{out}=0$: this effect is due to the inertial terms in the problem formulation, thus highlighting their importance in the model.

All the observations above are in general agreement with measurements and computational results in the literature, particularly in the magnitude of the displacement $\vec{d}$ and in the pressure difference between the third ventricle and the outflow \cite{csfGeneration,causemann2022human}.
Some discrepancies can be observed in the overall range of $p_\EE$ and the magnitude of $\vec{u}$, but they do not exceed one order of magnitude.
These discrepancies may be due to the fact that the computational domain of the simulation presented here is a 2D slice of the brain, and does not include the lateral ventricles:
simulations in three dimensions, capturing the complete geometry of the brain, are envisaged to address this issue.

\begin{figure}
    \centering
    \includegraphics[width=0.8\textwidth]{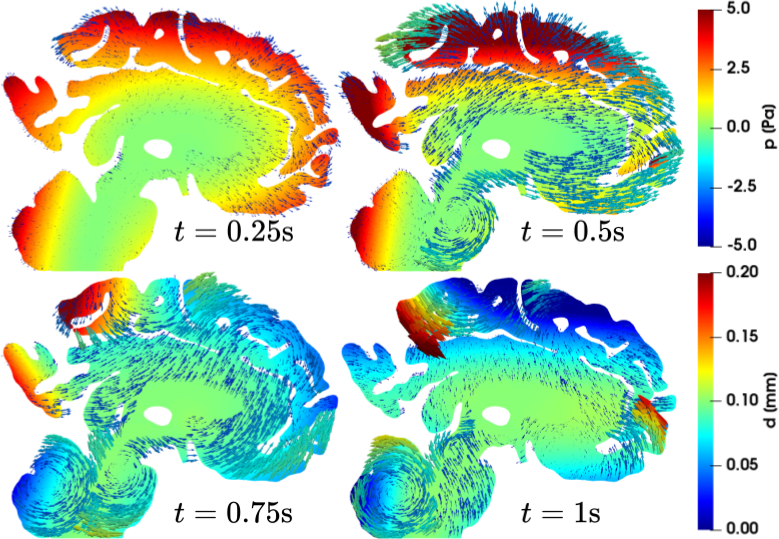}
    \caption{Brain simulation of \cref{sec:brain}. Discrete tissue displacement $\vec{d}_h$ and pressures $p_{\EE,h},p_h$ in the whole domain $\Omega_\PP\cup\Omega_\FF$ at different times during a heartbeat. Same scale for $p_{\EE,h},p_h$.}
    \label{fig:brainDP}
\end{figure}

\begin{figure}
    \centering
    \includegraphics[width=0.7\textwidth]{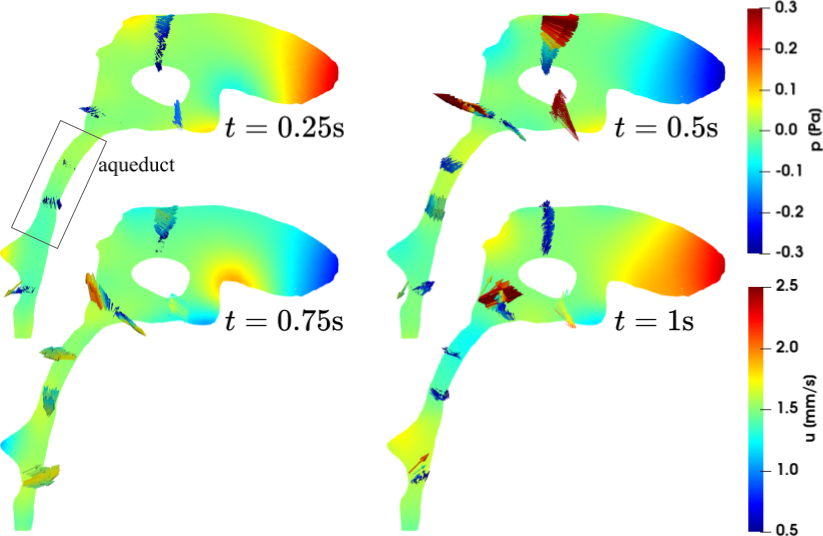}
    \caption{Brain simulation of \cref{sec:brain}. Velocity $\vec{u}_h$ and pressure $p_h$ in the brain ventricles $\Omega_\FF$ at different times during a heartbeat.
    The cerebral aqueduct is indicated by the box.}
    \label{fig:brainUP}
\end{figure}

\section{Conclusions}
We introduced a discontinuous Galerkin polytopal method for the discretization of a multiphysics fluid-mechanics model of the brain, encompassing a Multi-compartment Poro-Elastic (MPE) modeling of the tissue, perfused by blood and CSF, and Stokes' flow of the CSF in the brain cavities.
Our numerical method is particularly suitable for capturing the high geometrical complexity of the brain, thanks to the possibility of supporting high-order approximations and agglomerated grids.
We proved stability and optimal error bounds for the semidiscrete formulation under standard assumptions in the PolyDG framework.
Specific attention was paid to the modeling and numerical treatment of the interface conditions between the MPE and fluid domain.
We implemented the method in our PolyDG library \lymph{} \cite{lymph}.

We performed verification tests, validating the optimal order of convergence of our method with respect to the mesh element size and spectral convergence with respect to the polynomial degree.
Then, we showed the capability of our computational model to represent the physiological clearance function of the CSF in physiological settings, on a two-dimensional sagittal slice of the brain and of its ventricles.

Several directions for further development of the present work should be addressed.
First, in terms of modeling, nonlinear hyperelastic rheology could be considered to account for the extremely soft nature of the brain tissue \cite{mihai2017family,budday2017rheological,anssari2022modelling}.
This would require more detailed fluid-structure interaction conditions on the tissue-fluid interface, which have not been fully studied in the brain ventricles, but have been widely investigated in other biological systems such as the heart and blood vessels 
\cite{quarteroni2017cardiovascular,feng2019analysis,hirschhorn2020fluid,bucelli2023mathematical,fumagalli2023fluid}.
Second, to thoroughly analyze the effect of the complex brain geometry on waste clearance, the model could be applied to full three-dimensional geometries of different subjects.
To this aim, suitable geometry reconstruction techniques based on MRI data should be employed, such as those developed in \cite{ringstad2018brain,mardal2022mathematical,su2016cardiac,fumagalli2020image,fumagalli2022image,bennati2023image,baenen2023energetics}.
Third, to better represent the effect of blood pulsation on CSF flow (here modeled by a homogeneous CSF generation term), the blood compartments of the MPE model should be coupled to models of the main brain arteries \cite{jones2021anatomical,liu2020state}.
Finally, the coupling of the model proposed here with a suitable description of the generation, aggregation, and transport of misfolded proteins such as amyloid-$\beta$ and tau, would allow the investigation of their clearance in different physiological and pathological conditions \cite{corti2023structure,corti2023discontinuous}.
This would entail the development of numerical methods able to account for very different time scales, since the model presented here is defined on the scale of seconds (characteristic of blood pulsation) while prion aggregation and neurodegeneration occur over a timespan of decades.

\appendix

\section{Continuity and coercivity in the discrete spaces}\label{sec:contcoerc}
We collect the continuity and coercivity results of \cite{corti2022numerical,AMVZ22} in the following Lemma, which is instrumental to the proofs of stability (\cref{th:stab}) and convergence (\cref{th:conv}) of our numerical method \eqref{eq:DG}.
\begin{lemma}\label{th:contcoerc}
    Under \cref{hp:mesh}, the forms of \eqref{eq:formsSeparated} are continuous over the discrete spaces:
    \[
    \begin{aligned}
    |\mathcal A_\PP(\vec{d},\vec{w})| &\lesssim \normDGd{\vec{d}}\normDGd{\vec{w}} && \forall \vec{d},\vec{w}\in\spaceW,\\
    |\mathcal A_\jj(p_\jj,q_\jj)| & \lesssim
    \normDGj{p_\jj}\normDGj{q_\jj} && \forall p_\jj,q_\jj\in\spaceQj,\quad \jj\in J,\\
    \left|\sum_{\jj\in J} \mathcal C_\jj(\{p_\kk\}_{\kk\in J}, q_\jj)\right| & \lesssim \sum_{\jj,\kk\in J}\normDGk{p_\kk}\normDGj{q_\jj} && \forall p_\kk,\in\spaceQk, q_\jj\in\spaceQj,\quad \jj,\kk\in J,\\
    |\mathcal B_\jj(q_\jj,\vec{w})| & \lesssim \normDGj{q_\jj}\normDGd{\vec{w}} && \forall \vec{w}\in\spaceW, q_\jj\in\spaceQj,\quad \jj\in J,\\
    |\mathcal A_\FF(\vec{u},\vec{v})| &\lesssim \normDGu{\vec{u}}\normDGu{\vec{v}} && \forall \vec{u},\vec{v}\in\spaceV,\\
    |\mathcal B(q,\vec{v})| & \lesssim \normDGp{q}\normDGu{\vec{v}} && \forall \vec{v}\in\spaceV, q\in\spaceQ,
    \end{aligned}
    \]
    Moreover,
    provided that the penalty constants are chosen sufficiently large,
    the following coercivity and inf-sup inequalities hold (for any $\jj\in J$):
    \[
    \begin{aligned}
    &\mathcal A_\PP(\vec{w},\vec{w}) \gtrsim \normDGd{\vec{w}}^2 \qquad \forall \vec{w}\in\spaceW,\\
    &\mathcal A_\FF(\vec{v},\vec{v}) \gtrsim \normDGu{\vec{v}}^2 \qquad\forall \vec{v}\in\spaceV,\\
    &\sum_{\jj\in J} \mathcal C_\jj(\{q_\kk\}_{\kk\in J}, q_\jj) \gtrsim \sum_{\jj\in J}\|\sqrt{\beta_\jj^\text{e}}q_\jj\|_{\text{DG},P_\jj}^2 
    \quad\text{and}\quad
    \mathcal A_\jj(q_\jj,q_\jj)  \gtrsim
    \normDGj{q_\jj}^2 \quad \forall q_\jj\in\spaceQj,\ \jj\in J,
    \\
    &\sup_{\vec{v}\in \spaceV\setminus\{0\}} \frac{\mathcal B_\FF(q,\vec{v})}{\normDGu{\vec{v}}} + \sqrt{\mathcal S(q,q)} \geq \beta_{\FF,h}\|q\|_{\Omega_\FF} \qquad \forall q\in\spaceQ,
    \end{aligned}
    \]
    where $\beta_{\FF,h}$ is the discrete inf-sup constant \cite{AMVZ22}.
\end{lemma}

\section{Proof of \cref{th:newcontJ}}\label{sec:proofnewcontJ}
    Let $q\in H^2(\mathscr T_{h,\PP}), \vec{w}\in H^2(\mathscr T_{h,\PP}), \vec{v}\in H^2(\mathscr T_{h,\FF})$,
    and $q\in\spaceQE, \vec{w}_h\in \spaceW, \vec{v}_h\in\spaceV$.
    For any $F\in\mathscr F_h^\Sigma$, we denote by $K_\PP^F$ and $K_\FF^F$ the elements of $\mathscr T_{h,\PP}$ and $\mathscr T_{h,\FF}$, respectively, sharing $F$ in their boundary. Then,
    \[
    \begin{aligned}
    |\mathcal J(q,\vec{w}_h,\vec{v}_h)| & 
    \leq \sum_{F\in \mathscr F_h^\Sigma} \left(\left|\int_F q\,\vec{w}_h\cdot\vec{n}_\PP\right|+\left|\int_F q\,\vec{v}_h\cdot\vec{n}_\FF\right|\right)
    \\ &
    \lesssim \sum_{F\in\mathscr F_h^\Sigma} \left(\|\eta^{1/2}q\|_F\|\eta^{-1/2}\vec{w}_h\|_F + \|\gamma_p^{-1/2}q\|_F\|\gamma_p^{1/2}\vec{v}_h\|_F\right)
    \\ &
    \overset{\fbox{TR}}{\lesssim} \sum_{F\in\mathscr F_h^\Sigma} \left( \|\eta^{1/2}q\|_F\cancel{h_{K_\PP^F}^{1/2}}\cancel{{h_{K_\PP^F}^{-1/2}}}\|\vec{w}_h\|_{K_\PP^F} + \|\gamma_p^{-1/2}q\|_F\cancel{h_{K_\FF^F}^{1/2}}\cancel{{h_{K_\FF^F}^{-1/2}}}\|\vec{v}_h\|_{K_\FF^F} \right)
    \\ &
    \leq \|\eta^{1/2}q\|_{\mathscr F_h^\Sigma}\normDGd{\vec{w}_h} + \|\gamma_p^{-1/2}q\|_{\mathscr F_h^\Sigma}\normDGu{\vec{v}_h},
    \end{aligned}
    \]
where, in line \fbox{TR}, we have used the inverse trace inequality \eqref{eq:inversetrace} and the fact that
$\eta|_K\sim h_K^{-1} \ \forall K\in \mathscr T_{h,\PP}$ and $\gamma_p|_K\sim h_K \ \forall K\in\mathscr T_{h,\FF}$.
    An analogous argument (using $\gamma_\vec{v}^{1/2}$ in place of $\gamma_p^{-1/2}$) allows to control $|J(q_h,\vec{w},\vec{v})|$, thus concluding the proof.

\section{Proof of \cref{th:conv}}\label{sec:proofConv}

In this section, we prove the optimal convergence estimate stated in \cref{th:conv}.
This result is based on the inverse trace inequality \eqref{eq:inversetrace}, on \cref{th:newcontJ} and on the following continuity results (proved in \cite{antonietti2022high,antonietti2023discontinuous,AMVZ22}), which extend \cref{th:contcoerc} to consider also non-discrete functions:
\begin{lemma}\label{th:newcontOthers}
    Under the same assumptions of \cref{th:contcoerc}, the following inequalities hold:
    \[
    \begin{aligned}
    |\mathcal A_\PP(\vec{d},\vec{w}_h)| &\lesssim \normcontD{\vec{d}}\normDGd{\vec{w}_h} && \forall \vec{d}\in[H^2(\mathscr T_{h,\PP})]^d, \vec{w}_h\in\spaceW,\\
    |\mathcal A_\jj(p_\jj,q_{\jj,h})| & \lesssim
    \normcontJ{p_\jj}\normDGj{q_{\jj,h}} && \forall p_\jj \in H^2(\mathscr T_{h,\PP}),q_{\jj,h}\in\spaceQj,\quad \forall\jj\in J,\\
    |\mathcal B_\jj(q_{\jj,h},\vec{w})| & \lesssim \normDGj{q_{\jj,h}}\normcontD{\vec{w}} && \forall \vec{w}\in [H^2(\mathscr T_{h,\PP})]^d, q_{\jj,h}\in\spaceQj,\quad \forall\jj\in J,\\
    |\mathcal B_\jj(q_\jj,\vec{w}_h)| & \lesssim \normcontJ{q_\jj}\normDGd{\vec{w}_h} && \forall \vec{w}_h\in\spaceW, q_\jj\in H^2(\mathscr T_{h,\PP}),\quad \forall\jj\in J,\\
    |\mathcal A_\FF(\vec{u},\vec{v}_h)| &\lesssim \normcontU{\vec{u}}\normDGu{\vec{v}_h} && \forall \vec{u}\in [H^2(\mathscr T_{h,\FF})]^d,\vec{v}_h\in\spaceV,\\
    |\mathcal B_\FF(q_h,\vec{v})| & \lesssim \normDGp{q_h}\normcontU{\vec{v}} && \forall \vec{v}\in [H^2(\mathscr T_{h,\FF})]^d, q_h\in\spaceQ,\\
    |\mathcal B_\FF(q,\vec{v}_h)| & \lesssim \normcontP{q}\normDGu{\vec{v}_h} && \forall \vec{v}_h\in\spaceV, q\in H^1(\mathscr T_{h,\FF}).
    \end{aligned}
    \]
\end{lemma}

We are now ready to prove our optimal convergence result.
\begin{proof}(\cref{th:conv})
    Using the interpolators defined in \cref{th:interp}, we split the error $\vec{e}^\vec{d}=\vec{e}^\vec{d}_I-\vec{e}^\vec{d}_h$ into an interpolation error $\vec{e}^\vec{d}_I=\vec{d}-\vec{d}_I\in[H^{m+1}(\mathscr T_{h,\PP})]^d$ and an approximation error $\vec{e}^\vec{d}_h=\vec{d}_h-\vec{d}_I\in\spaceW$.
    An analogous notation is introduced for the errors $e^{P_\jj}\ \forall\jj\in J, \vec{e}^\vec{u}, e^{P_\FF}$.
    We fix a time $t\in(0,T]$ and we proceed similarly to \cite{corti2022numerical}, namely we subtract equation \eqref{eq:DG} from \eqref{eq:weak}, tested against $(\partial_t\vec{e}^\vec{d}_h,\{e^{P_{\jj}}_h\}_{\jj\in J},\vec{e}^\vec{u}_h,e^{P_\FF}_h)$, thus obtaining
    \[\begin{aligned}
    (\rho_\PP\partial_{tt}^2\vec{e}^\vec{d}&,\partial_t\vec{e}^\vec{d}_h)_{\Omega_\PP} + \mathcal A_\PP(\vec{e}^\vec{d},\partial_t\vec{e}^\vec{d}_h) + \sum_{\kk\in J}\mathcal B_\kk(e^{P_\kk},\partial_t\vec{e}^\vec{d}_h) \\
    & + \sum_{\jj\in J} \left[ (c_\jj\partial_t e^{P_\jj}, e^{P_\jj}_h)_{\Omega_\PP} + \mathcal A_\jj(e^{P_\jj},e^{P_\jj}_h) + \mathcal C_\jj(\{e^{P_\kk}\}_{\kk\in J},e^{P_\jj}_h) - \mathcal B_\jj(e^{P_\jj}_h,\partial_t \vec{e}^\vec{d}) \right] \\
    & + (\rho_\FF\partial_t\vec{e}^\vec{u},\vec{u}_h)_{\Omega_\FF} + \mathcal A_\FF(\vec{e}^\vec{u},\vec{e}^\vec{u}_h) + \mathcal B_\FF(e^{P_\FF},\vec{e}^\vec{u}_h)  - \mathcal B_\FF(e^{P_\FF}_h,\vec{e}^\vec{u}) + \mathcal S(e^{P_\FF},e^{P_\FF}_h)\\
    & + \mathcal J(e^{P_\EE}, \partial_t\vec{e}^\vec{d}_h,\vec{e}^\vec{u}_h) - \mathcal J (e^{P_{\EE}}_h, \partial_t\vec{e}^\vec{d},\vec{e}^\vec{u}) = 0.
    \end{aligned}\]
    According to the splitting introduced above, we separate the terms depending only on the discrete approximation errors from those involving the interpolation errors.
    Then,
    integrating in time from $0$ to $t$ and proceeding as in the proof of \cref{th:stab},
    the coercivity inequalities of \cref{th:contcoerc} yield
    \begin{equation}\label{eq:proofconv1}
    \begin{aligned}
    \|&\sqrt{\rho_\PP}\partial_t\vec{e}^\vec{d}_h\|_{\Omega_\PP}^2 + \normDGd{\vec{e}^\vec{d}_h}^2
    + \sum_{\kk\in J}\left[\|\sqrt{c_\kk}e^{P_\kk}_h\|_{\Omega_\PP}^2 + \int_0^t\left(\normDGk{e^{P_\kk}_h}^2 + \|\sqrt{\beta^e_\kk}e^{P_\kk}_h\|_{\Omega_\PP}^2\right)ds \right]
    \\ &\qquad
    +\|\sqrt{\rho_\FF}\vec{e}^\vec{u}_h\|_{\Omega_\FF}^2 + \int_0^t\alpha\left(\normDGu{\vec{e}^\vec{u}_h}^2+\normDGp{e^{P_\FF}_h}^2\right)ds
    \\ &\qquad
    + \int_0^t\left[\cancel{J(e^{P_\EE}_h, \partial_t\vec{e}^\vec{d}_h,\vec{e}^\vec{u}_h)} - \cancel{J(e^{P_\EE}_h, \partial_t\vec{e}^\vec{d}_h,\vec{e}^\vec{u}_h)}
    \right]ds\\ &
    \lesssim \int_0^t(\rho_\PP\partial_{tt}^2\vec{e}^\vec{d}_I,\partial_t\vec{e}^\vec{d}_h)_{\Omega_\PP}ds + \mathcal A_\PP(\vec{e}^\vec{d}_I,\vec{e}^\vec{d}_h) - \int_0^t\mathcal A_\PP(\partial_t\vec{e}^\vec{d}_I,\vec{e}^\vec{d}_h)ds
    \\ &\qquad
    + \sum_{\kk\in J}\int_0^t\left[(c_\kk \partial_t e^{P_\kk}_I,e^{P_\kk}_h)_{\Omega_\PP} + \mathcal A_\kk(e^{P_\kk}_I,e^{P_\kk}_h) + \mathcal C_\kk(\{e^{P_\jj}_I\}_{\jj\in J},e^{P_\kk}_h) \right]ds
    \\ &\qquad
    + \sum_{\kk\in J}\mathcal B_\kk(e^{P_\kk}_I,\vec{e}^\vec{d}_h) - \sum_{\kk\in J}\int_0^t\left[\mathcal B_\kk( e^{P_\kk}_I,\partial_t\vec{e}^\vec{d}_h) + \mathcal B_\kk(e^{P_\kk}_h,\partial_t\vec{e}^\vec{d}_I)\right]ds
    \\ &\qquad
    + \int_0^t\left[(\rho_\FF\partial_t \vec{e}^\vec{u}_I,\vec{e}^\vec{u}_h)_{\Omega_\FF} + \mathcal A_\FF(\vec{e}^\vec{u}_I,\vec{e}^\vec{u}_h) + \mathcal B_\FF(e^{P_\FF}_I,\vec{e}^\vec{u}_u) - \mathcal B_\FF(e^{P_\FF}_h,\vec{e}^\vec{u}_I) + \mathcal S(e^{P_\FF}_I,e^{P_\FF}_h)\right]ds
    \\ &\qquad
    + \int_0^t\left[J(e^{P_\EE}_I, \partial_t\vec{e}^\vec{d}_h,\vec{e}^\vec{u}_h) - J(e^{P_\EE}_h, \partial_t\vec{e}^\vec{d}_I,\vec{e}^\vec{u}_I)\right]ds,
    \end{aligned}
    \end{equation}
    where no contribution arises from the initial conditions, due to the hypotheses.
    Using \cref{th:newcontJ,th:newcontOthers} and the inequality 
    \[
    \sum_{\kk\in J}|C_\kk(\{e_I^{P_\jj}\}_{\jj\in J}. e^{P_\kk}_h)|\lesssim \sum_{\jj,\kk\in J}\|c_\jj^{-1/2}e_I^{P_\jj}\|_{\Omega_\PP}\|\sqrt{c_\kk}e_h^{P_\kk}\|_{\Omega_\PP} \lesssim \sum_{\jj,\kk\in J}\normcontJ{e_I^{P_\jj}}\|\sqrt{c_\kk}e_h^{P_\kk}\|_{\Omega_\PP}
    \]
    
    on the right-hand side of \eqref{eq:proofconv1} yields
    \begin{equation}\label{eq:proovconv2}\begin{aligned}
    &
    \normEN{(\vec{e}^\vec{d}_h,\{e^{P_\jj}_h\}_{\jj\in J},\vec{e}^\vec{u}_h,e^{P_\FF}_h)}^2
    \\ &\qquad
    \lesssim \normcontD{\vec{e}^\vec{d}_I(t)}\normDGd{\vec{e}^\vec{d}_h(t)}
    + \int_0^t\left[
    \|\sqrt{\rho_\PP}\partial_{tt}^2\vec{e}^\vec{d}_I(s)\|_{\Omega_\PP}\|\sqrt{\rho_\PP}\partial_t\vec{e}^\vec{d}_h(s)\|_{\Omega_\PP} 
    \right.\\&\qquad\qquad\qquad\qquad\qquad\qquad\qquad\qquad\qquad\left.+ \normcontD{\partial_t\vec{e}^\vec{d}_I(s)}\normDGd{\vec{e}^\vec{d}_h(s)}
    \right]ds
    \\ &\qquad\qquad
    + \sum_{\kk\in J}\left[
    \normcontK{e^{P_\kk}_I(t)}\normDGd{\vec{e}^\vec{d}_h(t)}
    + \int_0^t
    \|\sqrt{c_\kk}\partial_t e^{P_\kk}_I(s)\|_{\Omega_\PP}\|\sqrt{c_\kk}e^{P_\kk}_h(s)\|_{\Omega_\PP}
    ds\right]
    \\ &\qquad\qquad
    + \sum_{\kk\in J} \int_0^t\left[
    \normcontK{e^{P_\kk}_I(s)}\normDGk{e^{P_\kk}_h(s)}
    + \normcontK{e^{P_\kk}_I(s)}\normDGd{\partial_t \vec{e}^\vec{d}_h(s)}
    \right]ds
    \\ &\qquad\qquad
    + \sum_{\kk\in J} \int_0^t\left[
    \normDGk{e^{P_\kk}_h(s)}\normcontD{\partial_t\vec{e}^\vec{d}_I(s)}
    + \sum_{\jj\in J}\normcontJ{e^{P_\jj}_I(s)}\|\sqrt{c_\kk}e^{P_\kk}_h(s)\|_{\Omega_\PP}
    \right]ds
    \\ &\qquad\qquad
    + \int_0^t[
    \|\sqrt{\rho_\FF}\partial_t \vec{e}^\vec{u}_I(s)\|_{\Omega_\FF}\|\vec{e}^\vec{u}_h(s)\|_{\Omega_\FF} + \normcontU{\vec{e}^\vec{u}_I(s)}\normDGu{\vec{e}^\vec{u}_h(s)}
    \\&\qquad\qquad\qquad\qquad\qquad\qquad\qquad\qquad\qquad
    + \normcontP{e^{P_\FF}_I(s)}\normDGu{\vec{e}^\vec{u}_h(s)}
    ]ds
    \\ &\qquad\qquad
    + \int_0^t\left[
    \normDGp{e^{P_\FF}_h(s)}\normcontU{\vec{e}^\vec{u}_I(s)} + \sqrt{\mathcal S(e^{P_\FF}_I(s),e^{P_\FF}_I(s))}\sqrt{\mathcal S(e^{P_\FF}_h(s),e^{P_\FF}_h(s))}
    \right]ds
    \\ &\qquad\qquad
    + \int_0^t\left[
    \|\eta^{1/2}e^{P_\EE}_I\|_{\mathscr F_h^\Sigma}\normDGd{\partial_t\vec{e}^\vec{d}_h} + \|\gamma_p^{-1/2}e^{P_\EE}_I\|_{\mathscr F_h^\Sigma}\normDGu{\vec{e}^\vec{u}_h}
    \right]ds
    \\ &\qquad\qquad
    + \int_0^t\left[
    \normDGE{e^{P_\EE}_h}\left(\|\eta^{1/2}\partial_t\vec{e}^\vec{d}_I\|_{\mathscr F_h^\Sigma} + \|\gamma_\vec{v}^{1/2}\vec{e}^\vec{u}_I\|_{\mathscr F_h^\Sigma}\right)
    \right]ds.
    \end{aligned}\end{equation}
Regarding the interface terms, we can observe that the choice of $\eta,\gamma_\vec{v},\gamma_p$ made in \eqref{eq:penaltyparams} implies that $\eta,\gamma_\vec{v},\gamma_p^{-1}$ scale as $h_K$ in each element $K\in\mathscr T_h$.
Thus, thanks to \cref{th:interp}, we obtain
\begin{align*}
\|\eta^{1/2}e_I^{P_\EE}\|_{\mathscr F_h^\Sigma}
&
\lesssim \sum_{K\in \mathscr T_{h,\PP}}\|\eta\|_{L^\infty(K)}^{1/2}\|h_K^{m+1/2}\|\mathcal E_K p_\EE\|_{H^{m+1}(\widehat{K})}
\lesssim \sum_{K\in \mathscr T_{h,\PP}}h_K^m\|\mathcal E_K p_\EE\|_{H^{m+1}(\widehat{K})},
\end{align*}
and similar optimal estimates for the other interpolation errors at the interface appearing in the last two lines of \eqref{eq:proovconv2}.

By Cauchy-Schwarz's and Young's inequalities, all the terms involving the discrete errors $\vec{e}^\vec{d}_h,e^{P_\jj}_h,\vec{e}^\vec{u}_h,e^{P_\FF}_h$ on the right-hand side of \eqref{eq:proovconv2} can be moved to the left-hand side, whereas the interpolation errors $\vec{e}^\vec{d}_I,e^{P_\jj}_I,\vec{e}^\vec{u}_I,e^{P_\FF}_I$ can be controlled by the estimates of \cref{th:interp}, yielding
\begin{equation}\label{eq:proofconvLast}\begin{aligned}
    &
    \normEN{(\vec{e}^\vec{d}_h,\{e^{P_\jj}_h\}_{\jj\in J},\vec{e}^\vec{u}_h,e^{P_\FF}_h)}^2
    \\ &\qquad
    \lesssim \sum_{K\in\mathscr T_{h,\PP}}
    \frac{h_K^{2m}}{m}\left\{
    \|\mathcal E_K\vec{d}(t)\|_{[H^{m+1}(\widehat{K})]^d}^2 + \sum_{\kk\in J}\|\mathcal E_K p_\kk(t)\|_{H^{m+1}(\widehat{K})}^2
    \right.
    \\ &\qquad
    \qquad\quad\left.
    \qquad\quad+\int_0^t\left[
    \|\mathcal E_K\partial_t\vec{d}(s)\|_{[H^{m+1}(\widehat{K})]^d}^2+\|\mathcal E_K\partial_{tt}^2\vec{d}(s)\|_{[H^{m+1}(\widehat{K})]^d}^2
    \phantom{\sum_{k\in J}}\right.
    \right.
    \\ &\qquad
    \qquad\qquad\left.
    \qquad\quad\left.
    +\sum_{k\in J}\left(
    \|\mathcal E_Kp_\kk(s)\|_{H^{m+1}(\widehat{K})}^2+\|\mathcal E_K\partial_tp_\kk(s)\|_{H^{m+1}(\widehat{K})}^2
    \right)
    \right]ds
    \right\}
    \\ &\qquad
    \quad+ \sum_{K\in\mathscr T_{h,\FF}}
    \frac{h_K^{2m}}{m}\left.
    \int_0^t\left[
    \|\mathcal E_K\vec{u}(s)\|_{[H^{m+1}(\widehat{K})]^d}^2 + \|\mathcal E_K\partial_t\vec{u}(s)\|_{[H^{m+1}(\widehat{K})]^d}^2
    \right.
    \right.
    \\ &\qquad
    \qquad\qquad\qquad\qquad\qquad\left.
    \left.
    +\|\mathcal E_Kp(s)\|_{H^{m+1}(\widehat{K})}^2
    \right]ds.
    \right.
\end{aligned}\end{equation}
Observing that an estimate for $\normEN{(\vec{e}^\vec{d}_I,\{e^{P_\jj}_I\}_{\jj\in J},\vec{e}^\vec{u}_I,e^{P_\FF}_I)}$ that is completely analogous to \eqref{eq:proofconvLast} can be proven by resorting to \cref{th:interp}, the triangle inequality
\[
\normEN{(\vec{e}^\vec{d},\{e^{P_\jj}\}_{\jj\in J},\vec{e}^\vec{u},e^{P_\FF})}^2 \leq \normEN{(\vec{e}^\vec{d}_h,\{e^{P_\jj}_h\}_{\jj\in J},\vec{e}^\vec{u}_h,e^{P_\FF}_h)}^2 + \normEN{(\vec{e}^\vec{d}_I,\{e^{P_\jj}_I\}_{\jj\in J},\vec{e}^\vec{u}_I,e^{P_\FF}_I)}^2
\]
concludes the proof.
\end{proof}

\section{Algebraic form of the fully discrete problem}\label{sec:matfullydiscrete}
The matrices $A_1,A_2$ of \eqref{eq:fullydiscr} have the following form:

\rotatebox{90}{\vbox{
\begin{align*}
A_1 &= \begin{bmatrix}
A_\PP & 0 & M_\PP & B_\text{A}^T & \cdots & B_\text{E}^T+J_{\PP}^T & 0 & 0 \\
0 & I & -\gamma\Delta t I & 0 & \cdots & 0 & 0 & 0 \\
-\frac{1}{\beta\Delta t^2}I & 0 & I & 0 & \cdots & 0 & 0 & 0 \\
-\frac{\theta\gamma}{\beta\Delta t}B_\text{A} & 0 & 0 & \theta K_{\text{A}\text{A}} & \cdots & \theta C_{\text{A}\text{E}} & 0 & 0\\
\vdots &\vdots &\vdots &\vdots &\vdots &\vdots &\vdots &\vdots\\
-\frac{\theta\gamma}{\beta\Delta t}(B_\EE + J_{\PP}) & & &  \theta C_{\EE\text{A}} & \cdots & \theta K_{\EE\EE} & -\theta J_{\FF} & 0 \\
0 & 0 & 0 & 0 & \cdots & \theta J_{\FF}^T & \frac{1}{\Delta t}M_\FF +\theta A_\FF & \theta B_\FF^T \\
0 & 0 & 0 & 0 & \cdots & 0 & -\theta B_\FF & \theta S
\end{bmatrix},\\[2em]
A_2 &= \begin{bmatrix}
0 & 0 & 0 & 0 & \cdots & 0 & 0 & 0 \\
0 & I & (1-\gamma)\Delta t I & 0 & \cdots & 0 & 0 & 0 \\
-\frac{1}{\beta\Delta t^2}I & -\frac{1}{\beta\Delta t}I & \frac{2\beta -1}{2\beta}I &  0 & \cdots & 0 & 0 & 0 \\
-\frac{\theta\gamma}{\beta\Delta t}B_\text{A} & \left(1-\frac{\theta\gamma}{\beta}\right) B_{\text{A}} & \theta\Delta t\left(1-\frac{\gamma}{2\beta}\right)B_\text{A} & \frac{1}{\Delta t} M_\text{A} - (1-\theta)K_{\text{A}\text{A}} & \cdots & (1-\theta)C_{\text{A}\text{E}} & 0 & 0\\
\vdots &\vdots &\vdots &\vdots &\vdots &\vdots\\
-\frac{\theta\gamma}{\beta\Delta t}(B_\EE + J_{\PP}) & \left(1-\frac{\theta\gamma}{\beta}\right)B_\EE & \theta\Delta t\left(1-\frac{\gamma}{2\beta}\right)B_\EE & (1-\theta)C_{\EE\text{A}} & \cdots & \frac{1  }{\Delta t}M_\EE + (1-\theta)K_{\EE\EE} & -(1-\theta) J_{\FF} & 0 \\
0 & 0 & 0 & 0 & \cdots & (1-\theta)J_{\FF}^T & \frac{1}{\Delta t} M_\FF + (1-\theta)A_\FF & (1-\theta)B_\FF^T \\
0 & 0 & 0 & 0 & \cdots & 0 & -(1-\theta)B_\FF & (1-\theta)S
\end{bmatrix}\\[2em]
\text{where} &\qquad
K_{\jj\jj} = 
A_\jj + C_{\jj\jj} \quad\forall\jj\in J.
\end{align*}
}}

\section*{Acknowledgment}
IF, NP, and PFA have been partially supported by ICSC--Centro Nazionale di Ricerca in High Performance Computing, Big Data, and Quantum Computing funded by European Union--NextGenerationEU.
PFA acknowledges the financial support by MUR under the PRIN 2017 research grant n. 201744KLJL.
IF, NP, and PFA have been partially funded by MUR for the PRIN 2020 research grant n. 20204LN5N5.
The present research is part of the activities of the project Dipartimento di Eccellenza 2023-2027, Dipartimento di Matematica, Politecnico di Milano.
All the authors are members of GNCS-INdAM.
\bibliography{main}
\bibliographystyle{apalike}

\end{document}